\documentclass[preprint,pdftex,english,11pt]{imsart}
\usepackage{amsthm,amsmath,amssymb,amsfonts,epic,multirow,ulem}

\usepackage{thmtools}

\usepackage{tgbonum}

\usepackage[numbers]{natbib}
\RequirePackage{natbib}

\usepackage{wrapfig}
\usepackage[english]{babel}
 \usepackage[T1]{fontenc} 

 \usepackage{makeidx}
 \usepackage{fancyhdr}
\usepackage{epsf}
\usepackage[dvips]{epsfig}  
\usepackage{subfig}
 \usepackage{graphicx}
\usepackage[nice]{nicefrac}
\usepackage{dsfont}
\usepackage{calrsfs} %Lettres plus jolies

\usepackage{pifont}
\usepackage{bm}

\usepackage{wasysym}
\usepackage{stmaryrd}
\usepackage{aeguill}
\usepackage{mathrsfs}
\usepackage{enumerate}
\usepackage[shortlabels]{enumitem}
\usepackage{enumitem}
\usepackage[dvipsnames,table,xcdraw]{xcolor}
\usepackage{cancel}
\usepackage{ulem}
\usepackage{caption}
\usepackage{tikz}

\usepackage{hyperref}
\definecolor{linkcolour}{rgb}{0,0.2,0.6}
\hypersetup{colorlinks,breaklinks,urlcolor=linkcolour, linkcolor=linkcolour, citecolor=linkcolour}

\usepackage{geometry}

\startlocaldefs

\newtheoremstyle{mytheoremstyle} % name
        {\topsep}                    % Space above
        {\topsep}                    % Space below
        {\itshape\fontfamily{ppl}\selectfont}                   % Body font
        {}                           % Indent amount
        {\fontfamily{ppl}\selectfont\bfseries\color{black}}                   % Theorem head font
        {.}                          % Punctuation after theorem head
        {.5em}                       % Space after theorem head
        {}  % Theorem head spec (can be left empty, meaning ‘normal’)
\theoremstyle{mytheoremstyle}

\newtheorem{theo}{Theorem}[section]
\newtheorem{defi}[theo]{Definition}
\newtheorem{prop}[theo]{Proposition}
\newtheorem{cor}[theo]{Corollary}
\newtheorem{assum}{Assumption}
\newtheorem{exam}{Example}

\newtheorem{lemma}[theo]{Lemma}
\newtheorem{remark}[theo]{Remark}

\newtheorem{fact}[theo]{Fact}

\makeatletter
\renewenvironment{proof}[1][\proofname]{\par
  \pushQED{\qed}%
  \fontfamily{ppl} \topsep6\p@\@plus6\p@\relax
  \trivlist
  \item[\hskip\labelsep\itshape\bfseries#1\@addpunct{.}]\ignorespaces}{%
  \popQED\endtrivlist\@endpefalse
}

\def\R{\mathbb{R}}
\def \N{\mathbb{N}}

\def \P{\mathbb{P}} % proba 

\def \E{\mathbb{E}} % esperance 
\newcommand{ \un }{\mathds{1}}

\newenvironment{merci}{\textbf{Acknowledgments.}}{ }
\endlocaldefs

\renewcommand{\P}{\mathbb{P}}

\newtheorem{postita}{Post-it}

\renewcommand{\P}{\mathbb{P}}

\makeatletter
\newcommand*\bigcdot{\mathpalette\bigcdot@{.5}}
\newcommand*\bigcdot@[2]{\mathbin{\vcenter{\hbox{\scalebox{#2}{$\m@th#1\bullet$}}}}}
\makeatother

\begin{document}

{\fontfamily{ppl}\selectfont

\begin{frontmatter}

%%%%%%%%%
%%%%%%%%%%

\title{The leftmost particle of branching subordinators}

\author{\fnms{Alexis} \snm{Kagan}\ead[label=e1]{alexis.kagan@auckland.ac.nz}}
\address{Department of Statistics, University of Auckland, New Zealand. \printead{e1}}

\and

\author{\fnms{Grégoire} \snm{Véchambre}\ead[label=e2]{vechambre@amss.ac.cn}}
\address{State Key Laboratory of Mathematical Sciences, Academy of Mathematics and Systems Science, Chinese Academy of Sciences, Beijing, China. \printead{e2}} \vspace{0.5cm}

\runauthor{Kagan,Véchambre}

%Version : \today{}

\runtitle{The leftmost particle of branching subordinators}

\begin{abstract}
We define a family of continuous-time branching particle systems on the non-negative real line, called branching subordinators, where particles move as independent subordinators. Each particle can also split (at possibly infinite rate) into several children (possibly infinitely many) whose positions relative to the position of the parent are random. These particle systems are in the continuity of branching L\'evy processes introduced by Bertoin and Mallein \cite{Bertoin_Mallein2019}. We pay a particular attention to the asymptotic behavior of the leftmost particle of branching subordinators. It turns out that, under some assumptions, the rate of growth of the position of the leftmost particle is of order $t^{\gamma}$ where $\gamma \in (0,1)$ depends explicitly on the parameters. This sub-linear growth is significantly different from the classical linear growth observed for regular branching random walks with non-negative displacements. 
\end{abstract}



\end{frontmatter}

%%%%%%%%
%%%%%%%%

\vspace{1cm}

\section{Introduction}\label{Intro}
We consider continuous-time particle systems on the non-negative real line called branching subordinators where particles move as independent subordinators. Each particle can also split (at possibly infinite rate) into several children (possibly infinitely many) whose positions relatively to the position of the parent are random. Branching subordinators can be seen as age-dependent branching processes, see for example \cite{Harris1}, and the end of this introduction for details. They are in the continuity of branching L\'evy processes introduced by Bertoin and Mallein \cite{Bertoin_Mallein2019}. One can think of branching L\'evy processes as continuous-time counterparts of branching random walks. They generalize very famous processes such as branching Brownian motion, continuous-time branching random walks (see \cite{Uchiyama_CTBRW} for instance) and processes bearing the same name, that is branching L\'evy processes in the sense of Kyprianou \cite{Kyprianou1999}. General branching L\'evy processes first appeared in a work of Bertoin \cite{Bertoin2016} in a less general setting than in \cite{Bertoin_Mallein2019} as a tool to study fragmentation processes, see the end of this introduction for details. Branching L\'evy processes are also fundamental in the construction of Self-Similar Markov trees introduced by Bertoin, Curien and Riera in the book \cite{BertoinCurienRiera}. Roughly, the latter objects are related to branching Lévy processes in the same way positive self-similar Markov processes are related to classical Lévy processes via Lamperti representation. \\
Although branching L\'evy processes can be used as a powerful tool to study various probabilistic models, we aim in the present paper to show that branching L\'evy processes and especially branching subordinators deserve a particular attention for themselves.

\vspace{0.2cm}

Introduce $\mathcal{Q}$, the space of point measures $\bm{\mu}$ on $[0,\infty)$ that are such that $\bm{\mu}([0,a]) \in \mathbb{Z}_+$ for any $a \geq 0$. Each element of $\mathcal{Q}$ can be identified with a non-decreasing sequence $\bm{x}=(x_n)_{n\geq 1}$ on $[0,\infty]$ such that $\lim_{n\to\infty}x_n=\infty$. The identification consists in setting $x_{n+1}=\ldots=x_{n+k}=a$ if $\bm{\mu}([0,a))=n$ and $\bm{\mu}(\{a\})=k$ (for $n\geq 0$, $k\geq 1$, $a\in [0,\infty)$), and for any $n\geq 1$, $x_n=\infty$ if $\bm{\mu}([0,\infty))<n$. We denote $\emptyset := (\infty,\ldots)$ and $(0,\emptyset):=(0,\infty,\ldots)$. We equip $\mathcal{Q}$ with the vague topology:
we say that a sequence $((y^n_1,y^n_2,\ldots))_n$ of elements of $\mathcal{Q}$ converges to $(y_1,y_2,\ldots)\in\mathcal{Q}$ if and only if, for any continuous function $f:[0,\infty]\to\R$ with compact support on $[0,\infty)$ and $f(\infty)=0$
 \begin{align}\label{ConvPP}
     \sum_{k\geq 1}f\big(y^n_k\big)\underset{n\to\infty}{\longrightarrow}\sum_{k\geq 1}f(y_k).
 \end{align}
For any $\bm{x}=(x_n)_{n\geq 1}\in\mathcal{Q}$ and $\mathfrak{z}\in[0,\infty)$, introduce $\tau_{\mathfrak{z}}\bm{x}:=(x_n+\mathfrak{z})_{n\geq 1}$. A branching subordinator $Y=(Y(t))_{t\geq 0}$ is a $\mathcal{Q}$-valued process starting from $(0,\emptyset)$ which satisfies the following branching property: if we denote by $\mathcal{F}^{Y}:=(\mathcal{F}^Y_t)_{t\geq 0}$ the natural filtration of $Y$, then for any $t,s\geq 0$, setting $\bm{y}=(y_n)_{n\geq 1}=Y(t)$, 
\begin{align}\label{BranchingProperty}
    Y(t+s)\overset{(\textrm{d})}{=}\sum_{j\geq 1}\tau_{y_j}Y^{(j)}(s),
\end{align}
where $(Y^{(j)}(s)$;\; $j\geq 1$) is a collection of\textrm{ i.i.d }copies of $Y(s)$ and independent of $\mathcal{F}^Y_t$. A stronger version of the branching property \eqref{BranchingProperty} is actually verified by $Y$, see Remark \ref{ReguBS} in Appendix \ref{construction}. \\
Similarly as for a classical subordinator, a branching subordinator is characterized by a couple $(d,\Lambda)$ where $d\geq 0$ and $\Lambda$ is a measure on $\mathcal{Q}$ satisfying 
\begin{align}
\Lambda(\{(0,\emptyset)\})+\Lambda(\{\emptyset\})=0, \ \int_{\mathcal{Q}}(1\wedge x_1)\Lambda(\mathrm{d}\bm{x}) <\infty, \ \Lambda(\mathcal{Q}^{\star})\in(0,\infty], \label{condbranchinglp0}
\end{align}
with $\mathcal{Q}^{\star}:=\{\bm{x}\in \mathcal{Q}; x_2 < \infty\}$.
The measure $\Lambda$ encodes for both the structure of jumps and the structure of branchings in the associated particle system, while $d$ is the speed of the continuous component of the motion of particles. More precisely, when $\Lambda(\mathcal{Q}^{\star})<\infty$ (referred to as the case of branching subordinators with finite birth intensity in \cite{Bertoin_Mallein2019}), the dynamic of a branching subordinator with characteristics $(d,\Lambda)$ can be described as follows: there is initially one particle in the system denoted by $\varnothing$ and, during its lifetime, the particle $\varnothing$ moves and experiences branchings. Before the first branching occurs, $\varnothing$ moves according to a subordinator $\xi$ with drift $d$ and L\'evy measure the image of $\Lambda(\cdot\cap(\mathcal{Q}\setminus\mathcal{Q}^{\star}))$ by the projection $\bm{x}=(x_n)_{n\geq 1}\mapsto x_1$. Let $(T,\bm{X})=(T,(X_n)_{n\geq 1})$ be a $(0,\infty)\times\mathcal{Q}^{\star}$-valued random variable with law $\Lambda(\mathcal{Q}^{\star})e^{-t\Lambda(\mathcal{Q}^{\star})}\mathrm{d}t\otimes\Lambda(\cdot\cap\mathcal{Q}^{\star})/\Lambda(\mathcal{Q}^{\star})$ and independent of $\xi$. At time $T$, the first branching occurs: the particle $\varnothing$ makes a jump of size $X_1$ and simultaneously gives progeny to a (possibly infinite) set of children such that, at birth, the $i$-th offspring exists if and only if $X_{i+1}<\infty$ and is then located at $\xi(T-)+X_{i+1}$, $\xi(T-)$ denoting the pre-jump position of $\varnothing$. Starting from $\xi(T-)+X_1$, the particle $\varnothing$ performs according to an independent copy of $\xi$ until the next branching event occurs, according to an independent copy of $(T,\bm{X})$, and so on. In turn, each newborn particle moves and experiences branchings as described previously for $\varnothing$, independently of other particles in the system. Note that the assumption $\Lambda(\{(0,\emptyset)\})=0$ removes transitions that have no effect while $\Lambda(\{\emptyset\})=0$ says that there is no killing. The second part of \eqref{condbranchinglp0} ensures the summability of small jumps and is analogue to the integrability condition satisfied by L\'evy measures of classical subordinators. If $\Lambda(\mathcal{Q}^{\star})=0$, then there is no branching so we have a classical subordinator. Eliminating this degenerate case is the reason for assuming the third part of \eqref{condbranchinglp0}. Therefore, a particle never dies and we obtain a system of persisting particles such that the full trajectory of any particle is a subordinator with drift $d$ and L\'evy measure $\Lambda_1$, the image of $\Lambda$ by the projection $\bm{x}=(x_n)_{n\geq 1}\mapsto x_1$. These trajectories will be referred to as the \textit{canonical trajectories} in the following, see Definition \ref{DefCanoTraj}. \\
When $\Lambda(\mathcal{Q}^{\star})=\infty$, that is the case of reproduction at infinite rate, a branching subordinator with characteristics $(d,\Lambda)$ is obtained as a non-decreasing limit of  branching subordinators with finite birth intensity called truncated branching subordinators, see Definition \ref{DefBranchingSubor}. The following assumption plays a key role in our study and will in particular allow to ensure that the non-decreasing limit exits as a process on $\mathcal{Q}$:
\begin{align}
\forall a >0, \int_{\mathcal{Q}} \big|\sharp \{n \geq 1;\; x_n<a \}-1\big| \Lambda(\mathrm{d}\bm{x}) < \infty, \label{condbranchinglplight}
\end{align}
where for any set $B$, $\sharp B$ stands for the cardinal of $B$. We will sometimes need to strengthen \eqref{condbranchinglplight} into 
\begin{align}
\forall a >0, \int_{\mathcal{Q}} \big|\sharp \{n \geq 1;\; x_n <a \}-1\big|^2 \Lambda(\mathrm{d}\bm{x}) &< \infty. \label{condbranchinglplight2}
\end{align}
We construct the branching subordinator with characteristic couple $(d,\Lambda)$ for any $(d,\Lambda)$ with $d\geq 0$ and $\Lambda$ satisfying \eqref{condbranchinglp0} and \eqref{condbranchinglplight} in Appendix \ref{construction}. We see that thanks to equation \eqref{condbranchinglplight}, whether $\Lambda(\mathcal{Q}^{\star})$ is finite or not, $\mathcal{Q}^{\star}=\cup_{a>0}\{\bm{x}\in\mathcal{Q};\; x_2<a\}$ and $\Lambda(\{\bm{x}\in\mathcal{Q};\; x_2<a\})<\infty$ for any $a>0$ and this will be crucial in the construction. Let us stress on the fact that arguments presented in Appendix \ref{construction} are borrowed from the construction of branching L\'evy processes of Bertoin and Mallein \cite{Bertoin_Mallein2019} and Shi and Watson \cite{Shi_Watson}, but we provide a slight adaptation tailored to our case.
To index the set of particles in a branching subordinator $Y$, we use a generalization of the Ulam–Harris notation, introduced by Shi and Watson \cite{Shi_Watson}. For any $t\geq 0$, we then denote by $\mathcal{N}(t)$ the set of particles in $Y$ alive at time $t$ and for any particle $u\in\mathcal{N}(t)$ and $s\leq t$, $Y_u(s)$ denotes the position of $u$ at time $s$, if already born at time $s$. Otherwise, $Y_u(s)$ stands for the position of the most recent ancestor of $u$ alive at time $s$, see Appendix \ref{construction} for a rigorous definition. \\
In general, \eqref{condbranchinglplight} is not sufficient to ensure the well-definiteness of a branching L\'evy process. The fact that, in our case, trajectories are non-decreasing is crucial. Instead of \eqref{condbranchinglplight}, it is usually required (see equation (1.4) in \cite{Bertoin_Mallein2019}) that there exists $\theta\geq 0$ such that 
%\begin{align*}
%    \E\Big[\sum_{k\geq 1}e^{-\theta Y_k(1)}\Big]\in(0,\infty)
%\end{align*}
%or, equivalently
\begin{align}
    \int_{\mathcal{Q}} \left ( \sum_{k=2}^{\infty} e^{-\theta x_k} \right ) \Lambda(\mathrm{d}\bm{x}) < \infty, \label{condbranchinglp3}
\end{align}
where, by convention, $e^{-\theta \times \infty}=0$ for any $\theta \geq 0$. If $\mathcal{E}(\Lambda):=\{\lambda\geq0;\; \int_{\mathcal{Q}} \left ( \sum_{k=2}^{\infty} e^{-\lambda x_k} \right ) \Lambda(\mathrm{d}\bm{x})<\infty\}$, then \eqref{condbranchinglp3} implies that $[\theta,\infty)\subset\mathcal{E}(\Lambda)$. For any $\lambda\in\mathcal{E}(\Lambda)$, we define $\kappa(\lambda)$ by
\begin{align}
\kappa(\lambda):=d\lambda + \int_{\mathcal{Q}} \left ( 1-\sum_{k=1}^{\infty} e^{-\lambda x_k} \right ) \Lambda(\mathrm{d}\bm{x}), \label{laplaceexpobranchingsub}
\end{align}
and $\kappa(\lambda)=-\infty$ otherwise. When \eqref{condbranchinglp3} holds, then, by a slight adaptation of Theorem 1.1(ii) in \cite{Bertoin_Mallein2019}, a branching subordinator $Y=(Y_u(t);u\in\mathcal{N}(t),\;t\geq 0)$ with characteristics $(d,\Lambda)$ satisfies, for all $t\geq 0$ and any $\lambda\in\mathcal{E}(\Lambda)$, 
\begin{align*}
    \E\Big[\sum_{u\in\mathcal{N}(t)}e^{-\lambda Y_u(t)}\Big]=e^{-t\kappa(\lambda)}\in(0,\infty). 
\end{align*}
In that case, if, for any $\lambda\in\mathcal{E}(\Lambda)$, we denote 
\begin{align}
W_{\lambda}(t):=e^{t\kappa(\lambda)}\sum_{u\in\mathcal{N}(t)}e^{-\lambda Y_u(t)}, \label{defwlambdat}
\end{align}
then $(W_{\lambda}(t))_{t\geq 0}$ is a non-negative martingale thus converging to a random variable $W_{\lambda}(\infty)\in[0,\infty)$ as $t\to\infty$. Bertoin and Mallein proved (Theorem 1.1 in \cite{Bertoin_Mallein_BiggMart}) that $(W_{\lambda}(t))_{t\geq 0}$ is uniformly integrable if and only if $\lambda\kappa'(\lambda)>\kappa(\lambda)$ and $\int_{\mathcal{Q}}\sum_{k\geq 1}e^{-\lambda x_k}(\log\sum_{k\geq 1}e^{-\lambda x_k}-1)^+\Lambda(\mathrm{d}\bm{x})<\infty$, with $\mathfrak{z}^+=\max(\mathfrak{z},0)$. Otherwise, $W_{\lambda}(\infty)=0$ almost surely. This is a slightly different version of the Biggin's martingale convergence theorem (originally dedicated to branching random walks, see for instance \cite{Biggins1977} and \cite{Alsmeyer_Iksanov}) in branching L\'evy processes settings. Another important martingale for branching random walks known as the derivative martingale has been studied in the context of branching Lévy processes, see \cite{MS23}. \\
Note that having \eqref{condbranchinglp3} for some $\theta\geq 0$ is strictly stronger than having \eqref{condbranchinglplight}. In the present paper, unlike the case of regular branching random walks or branching L\'evy processes, we shall allow $\mathcal{E}(\Lambda)=\varnothing$, that is 
\begin{align}\label{noexpomom}
    \forall\;\lambda\geq 0,\; \int_{\mathcal{Q}} \left ( \sum_{k\geq 2}e^{-\lambda x_k} \right ) \Lambda(\mathrm{d}\bm{x})=\infty. 
\end{align}
\begin{remark}\label{Survival}
Note that assuming $\Lambda(\{\emptyset\})=0$ yields $\P(\exists\; t\geq 0: \mathcal{N}(t)=\varnothing)=0$, that is, the system survives almost surely. One can notice that even in the case where a particle can be killed, the population still survives with a positive probability when \eqref{noexpomom} holds.
\end{remark}
We would like to stress on the fact that assuming \eqref{condbranchinglplight} instead of \eqref{condbranchinglp3} (thus allowing the unusual case \eqref{noexpomom}) does not aim to improve the aesthetics of the definition of branching subordinators but allows the study of a various range of new continuous-time branching particle systems, see Sections \ref{rareinfbrintro} and \ref{manybinbrintro} for instance. Besides, as we will see in Section \ref{secmainres}, this allows to observe intriguing behaviors that cannot be observed under the assumption \eqref{condbranchinglp3} for the position $\underline{Y}(t):=\min_{u\in\mathcal{N}(t)}Y_u(t)$ of the leftmost particle of our branching subordinators.

As mentioned at the very beginning of the Introduction, branching L\'evy processes are closely related to fragmentation processes. The latter processes were initially studied by Kolmogorov \cite{KolmogorovFragmentation} and describe the evolution of particles randomly breaking over time. Each particle is characterized by its mass. Let us mention in particular homogeneous fragmentation \cite{BertoinHomogeneous} (see also \cite[Sec. 3.1]{Bertoin2006}) and its generalization, the compensated fragmentation introduced by Bertoin \cite{Bertoin2016}. The homogeneous fragmentation is related to branching subordinators while the compensated fragmentation is related to branching spectrally negative L\'evy processes. Informally, homogeneous fragmentation can be described as follows: a particle of initial mass $1$ is possibly subject to continuous erosion and breaks (at possibly infinite rate) into children particles (possibly infinitely many). At a splitting time, a particle gives progeny to a set of children with birth mass $\mathfrak{m}p_1,\mathfrak{m}p_2,\cdots$ where  $\sum_{i\geq 1}p_i\leq 1$ and $\mathfrak{m}$ is the mass of the particle just before splitting. Each newborn particle behaves similarly as its parent and independently of other particle in the system and so on. Let $Y$ be a branching subordinator with characteristics $(d,\Lambda)$ where $\Lambda$ satisfies \eqref{condbranchinglp0} and $\Lambda(\{\bm{x}\in \mathcal{Q}; \sum_{i \geq 1} e^{-x_i} > 1\})=0$ (this implies \eqref{condbranchinglp3} since we then have $\int_{\mathcal{Q}} ( \sum_{k=2}^{\infty} e^{-\theta x_k} ) \Lambda(\mathrm{d}\bm{x}) \leq \int_{\mathcal{Q}} (1-e^{-x_1}) \Lambda(\mathrm{d}\bm{x}) < \infty$). Then one can see that 
%we have $\sum_{u\in\mathcal{N}(t)}e^{-Y_u(t)}\leq 1$ almost surely, for any $t\geq 0$ and that 
the process $(e^{-Y_u(t)};\; u\in\mathcal{N}(t),\; t\geq 0)$ is a homogeneous fragmentation process with erosion coefficient $d$ and whose dislocation measure is the image measure of $\Lambda$ by $\bm{x} \mapsto (e^{-x_1},e^{-x_2},\ldots)$. Here, $e^{-Y_u(t)}$ denotes the mass of a particle $u$ alive at time $t$. We also refer to \cite{Dadoun2017} and \cite{Shi_Watson} for more details on links between fragmentation processes and branching L\'evy processes. Note that there exists another type of fragmentation called self-similar fragmentation for which homogeneous fragmentation corresponds to the special case of fragmentation with index of self-similarity $0$, see \cite{bertoinSelf-similar}. \\
%Self-similar fragmentation is also related to branching L\'evy processes but building rigorously the connection is more involved, see \cite{BertoinCurienRiera}. \\
One can notice that the largest particle of a homogeneous fragmentation process and the position of the leftmost particle of the associated branching subordinator are directly related: $\max_{u\in\mathcal{N}(t)}e^{-Y_u(t)}=e^{-\underline{Y}(t)}$. The behavior of the largest particle, or largest fragment, has been deeply studied in the past few years, see for instance \cite{BertoinAsymptotic} for the homogeneous fragmentation and \cite{Dadoun2017}, \cite{DyszewskiGantertJohnstonProchnoSchmid}, \cite{DJPP24} for self-similar fragmentation. \\

Another motivation to study branching subordinators comes from age-dependent branching processes, and our model can be seen as an extension of these branching processes, where the heigh of a particle at birth is interpreted as a birth time, the displacement of the newborn particle with respect to the parent being a delay between the reproduction time of the parent and the birth of the child, and the jumps of a particle between reproductions are interpreted as dormancy periods. In this interpretation, the original time of the process is an abstract local time allowing to run some underlying Poisson processes. We are particularly interested in determining possible behaviors in the case of non-explosion, as the explosion for heavy-tailed age-dependent branching processes has received attention by the past (see \cite{Amini_Devroye_Griffiths_Olver} and more recently in \cite{AKMS25}). Here the term \textit{explosion} is a terminology used in the study of age-dependent branching processes and means that infinitely many particles are born in finite time. The analogue, in our context, of that phenomenon is discussed in details in Section \ref{explosionsection}.

\subsection{Examples of branching subordinators} \label{exampleintro}

We now present three different and explicit examples of branching subordinators, the last two being in the case \eqref{noexpomom}.

\subsubsection{A simple example: the branching Poisson process} \label{branchingpoisson}

We consider the simple example of a branching Poisson process, namely, a particle system on $\mathbb{Z}_+$ where each particle moves one step right with rate $r>0$ and branches into two particles (both located at the parent's site) with rate $\rho>0$. Clearly, this particle system is the branching subordinator with characteristic couple $(0,\Lambda^{bp}_{r,\rho})$ where we have set $\Lambda^{bp}_{r,\rho}(\mathrm{d}\bm{x}):=r \delta_{(1,\infty,\ldots)} + \rho \delta_{(0,0,\infty,\ldots)}$. Note that for any $r,\rho >0$, the conditions \eqref{condbranchinglp0} and \eqref{condbranchinglp3} are satisfied by $\Lambda^{bp}_{r,\rho}$ and one has $\kappa(\lambda)=r(1-e^{-\lambda})-\rho$. This implies that the branching Poisson process is well-defined, see Appendix \ref{construction}. The interest of this example is that its simplicity will allow to make some results more explicit. Moreover, it will play a role in the proof of Proposition \ref{criteria} below.

\subsubsection{$\alpha$-stable trajectories with rare infinite branchings} \label{rareinfbrintro}

We consider the example of a two-parameters family of laws of branching subordinators for which branchings are relatively rare (in the sense that they occur at finite rate) but where, at each branching, infinitely many new particles are born. We also assume that the subordinators that give the law of trajectories between the branchings are stable. 

More precisely, for $\alpha \in (0,1)$, we define by $\nu_{\alpha}$ the normalized $\alpha$-stable L\'evy measure on $(0,\infty)$, that is $\nu_{\alpha}(\mathrm{d}\mathfrak{z}):=\frac{\alpha}{\Gamma(1-\alpha)} \frac{\mathrm{d}\mathfrak{z}}{\mathfrak{z}^{\alpha+1}}$. We denote by $\tilde \nu_{\alpha}$ the natural extension of $\nu_{\alpha}$ to $\mathcal{Q}$ that consists in choosing the coordinate $x_1$ according to $\nu_{\alpha}$ and setting the other coordinates at $\infty$. More precisely, $\tilde \nu_{\alpha}(\mathrm{d}\bm{x}):= \nu_{\alpha}(\mathrm{d}x_1) \otimes \delta_{\infty}(\mathrm{d}x_2) \otimes \delta_{\infty}(\mathrm{d}x_3) \otimes \ldots$ and, for $\beta \in (0,1)$, let $y^{\beta}:=(y^{\beta}_n)_{n\geq 1}$ be defined by $y^{\beta}_n:=(\log n)^{\beta}$. 
We set $\Lambda^{ri}_{\alpha,\beta}(\mathrm{d}\bm{x})=\tilde \nu_{\alpha}(\mathrm{d}\bm{x})+\delta_{y^{\beta}}(\mathrm{d}\bm{x})$. In other words, each trajectory branches at rate $1$, at a branching time, displacements are given by the sequence $y^{\beta}$, and each trajectory follows a normalized $\alpha$-stable subordinator between branching times. Note that for any $\alpha,\beta \in (0,1)$, the conditions \eqref{condbranchinglp0} and \eqref{condbranchinglplight2} are satisfied by $\Lambda^{ri}_{\alpha,\beta}$. This implies that the branching subordinator with characteristic couple $(0,\Lambda^{ri}_{\alpha,\beta})$ is well-defined, see Appendix \ref{construction}. Finally, note that the measure $\Lambda^{ri}_{\alpha,\beta}$ falls in the case \eqref{noexpomom}. 

\subsubsection{$\alpha$-stable trajectories with many binary branchings} \label{manybinbrintro}

We consider the example of a two-parameters family of laws of branching subordinators for which branchings are frequent (in the sense that they occur at infinite rate) and where each branching is a binary branching. 

More precisely, for $\alpha \in (0,1)$ we let $\tilde \nu_{\alpha}$ be as in Section \ref{rareinfbrintro}. For $\theta >1$, we let $\mu_{\theta}$ be the sigma-finite measure on $\mathbb{R}_+$ defined by $\mu_{\theta}(\mathrm{d}\mathfrak{z}):=e^{\mathfrak{z}^{\theta}} \mathds{1}_{\{\mathfrak{z}\geq 0\}}\mathrm{d}\mathfrak{z}$. We denote by $\tilde \mu_{\theta}$ the extension of $\mu_{\theta}$ to $\mathcal{Q}$ that consists in choosing the coordinate $x_1$ as $0$, the coordinate $x_2$ according to $\mu_{\theta}$, and setting the other coordinates at $\infty$. More precisely, $\tilde \mu_{\theta}(\mathrm{d}\bm{x}):= \delta_{0}(\mathrm{d}x_1) \otimes \mu_{\theta}(\mathrm{d}x_2) \otimes \delta_{\infty}(\mathrm{d}x_3) \otimes \delta_{\infty}(\mathrm{d}x_4) \otimes \ldots$ and we set $\Lambda^{mb}_{\alpha,\theta}(\mathrm{d}\bm{x})=\tilde \nu_{\alpha}(\mathrm{d}\bm{x})+\tilde \mu_{\theta}(\mathrm{d}\bm{x})$. Note that for any $\alpha \in (0,1)$ and $\theta>1$, the conditions \eqref{condbranchinglp0} and \eqref{condbranchinglplight2} are satisfied by $\Lambda^{mb}_{\alpha,\theta}$. This implies that the branching subordinator with characteristic couple $(0,\Lambda^{mb}_{\alpha,\theta})$ is well-defined, see Appendix \ref{construction}. 

The measure $\Lambda^{mb}_{\alpha,\theta}$ clearly falls in the case \eqref{noexpomom} so this example illustrates that, for branching subordinators, one can very well have only binary branchings and yet be in the case \eqref{noexpomom}, which would not be possible for branching random walks. This shows that the setting of branching subordinators is somehow more adapted than the setting of branching random walks to generate natural examples in the case \eqref{noexpomom} which, as we will see in Section \ref{secmainres}, allows for a larger variety of asymptotic behaviors, compared with the classical case \eqref{condbranchinglp3}.

\subsection{Some more definitions} \label{moredef}

Before going any further, we need to introduce a few more definitions and notations. For any $a>0$, we define the \textit{truncated Laplace exponent} $\kappa_a(\cdot)$ by setting, for any $\lambda\geq 0$,
\begin{align}
\kappa_a(\lambda):=d\lambda + \int_{\mathcal{Q}} \left ( 1-e^{-\lambda x_1}-\sum_{k=2}^{\infty} e^{-\lambda x_k} \mathds{1}_{\{x_k<a\}} \right ) \Lambda(\mathrm{d}\bm{x}). \label{laplaceexpobranchingsubtrunc}
\end{align}
We see from \eqref{condbranchinglp0} and \eqref{condbranchinglplight} that $\kappa_a(\lambda)$ is well-defined for every $a>0$ and $\lambda \geq 0$. We then have  
\begin{align}
\kappa_a(\lambda)=\phi(\lambda) -M_a(\lambda), \label{laplaceexpospecialcase}
\end{align}
where we have set 
\begin{align}
\phi(\lambda) := d\lambda + \int_{\mathcal{Q}} \left ( 1-e^{-\lambda x_1} \right ) \Lambda(\mathrm{d}\bm{x}), \ M_a(\lambda):=\int_{\mathcal{Q}} \left ( \sum_{k=2}^{\infty} e^{-\lambda x_k} \mathds{1}_{\{x_k<a\}} \right ) \Lambda(\mathrm{d}\bm{x}). \label{laplaceexpospecialcase2}
\end{align}
The function $\kappa_a(\cdot)$, $a>0$, has a natural interpretation. For any $\bm{x} =(x_i)_{i\geq 1}\in \mathcal{Q}$, we set $T_a:\bm{x}\in\mathcal{Q}\mapsto \bm{x}^a\in\mathcal{Q}$ to be the $a$-truncated operator on $\mathcal{Q}$, where $\bm{x}^a$ is defined to be the sequence $\bm{x}$ capped at $a$ such that $x_1$ remains unchanged, that is $(\bm{x}^a)_1:=x_1$ and for any $i\geq 2$, $(\bm{x}^a)_i:=x_i \mathds{1}_{\{x_i<a\}}+\infty \mathds{1}_{\{x_i\geq a\}}$. Let $\Lambda^a(\mathrm{d}\bm{x}):=\mathds{1}_{\{\bm{x}\neq (0,\emptyset)\}}(T_a \Lambda)(\mathrm{d}\bm{x})$, where $T_a \Lambda$ denotes the image measure of $\Lambda$ by $T_a$. Note that $\Lambda^a(\mathcal{Q}^{\star})=\Lambda(\{ \bm{x}\in\mathcal{Q}; x_2<a \})$ so $\lim_{a\to\infty}\Lambda^a(\mathcal{Q}^{\star})=\Lambda(\mathcal{Q}^{\star})$ which is positive by \eqref{condbranchinglp0}. Therefore, there exists $a_0(\Lambda)\geq 0$ such that $\Lambda^a(\mathcal{Q}^{\star})>0$ for all $a>a_0(\Lambda)$. Note that if $\Lambda$ satisfies \eqref{condbranchinglp0} and \eqref{condbranchinglplight}, then $\Lambda^a$ satisfies the conditions \eqref{condbranchinglp0} and \eqref{condbranchinglp3} for all $a>a_0(\Lambda)$. For a branching subordinator $Y$ with characteristics $(d,\Lambda)$ and $a>a_0(\Lambda)$, we denote by $Y^a$ the sub-system of $Y$ obtained as follows: at each branching event (represented by an $\bm{x} =(x_i)_{i\geq 1}\in \mathcal{Q}$ with $x_2<\infty$), the particle born with displacement $x_1$ is kept and for $i\geq 2$, the particles born with displacement $x_i\geq a$ are removed from the system, along with their future lineage. Then, as shown in Appendix \ref{construction}, $Y^a$ is a branching subordinator with characteristics $(d,\Lambda^a)$ and the corresponding Laplace exponent is $\kappa_a(\cdot)$, that is, for any $t,\lambda\geq 0$
\begin{align}\label{ExpEspTruncature}
    \E\Big[\sum_{u\in\mathcal{N}^a(t)}e^{-\lambda Y^a_u(t)}\Big]=e^{-t\kappa_a(\lambda)}\in(0,\infty), 
\end{align}
where, similarly as for $Y$, $\mathcal{N}^a(t)$ stands for the set of particles in $Y^a$ alive at time $t$ and for any particle $u\in\mathcal{N}^a(t)$ and $s\leq t$, $Y^a_u(s)$ denotes the position of $u$ at time $s$, if already born at time $s$. Otherwise, $Y^a_u(s)$ stands for the position of the most recent ancestor of $u$ alive at time $s$. We refer to Appendix \ref{construction} for a joint construction of $Y$ and $Y^a$ where $Y^a$ is introduced first in Definition \ref{TroncBranchingSubor} (and a justification of \eqref{ExpEspTruncature} is provided), $Y$ is subsequently constructed in Definition \ref{DefBranchingSubor}, and a slightly different version of the above restriction property is stated in Lemma \ref{restrictionprop}.
Also note that each canonical trajectory associated with $Y^a$ is a subordinator with Laplace exponent $\phi$, see below \eqref{condbranchinglp0} for a description of canonical trajectories and Definition \ref{DefCanoTraj} for a rigorous definition.

\subsection{Main results} \label{secmainres}

The main purpose of the present paper is to study the behavior of the leftmost particle of branching subordinators. This behavior is of particular interest to us because it turns out that, under our assumptions, for a branching subordinator $Y$ with characteristics $(d,\Lambda)$, the growth of $t\mapsto\underline{Y}(t)=\min_{u\in\mathcal{N}(t)}Y_u(t)$ can possibly display interesting sub-linear behaviors. \\
\subsubsection{Divergence of the leftmost particle} \label{explosionsection}
Introduce $\underline{Y}(\infty):=\lim_{t\to\infty}\underline{Y}(t)$. Note that $\underline{Y}(\infty)$ is well defined since the function $t\mapsto\underline{Y}(t)$ is non-decreasing. We start by giving a criterion in terms of the truncated Laplace exponent to determine whether $\underline{Y}(\infty)$ is finite or infinite. For this, we first note from \eqref{condbranchinglplight} and dominated convergence that, for any $0<a_1<a_2$, $M_{a_2}(\lambda)-M_{a_1}(\lambda)$ converges to $0$ as $\lambda$ goes to infinity, so the quantity 
\begin{align}
\mathcal{L}(d,\Lambda):=\lim_{\lambda \rightarrow \infty} \kappa_a(\lambda)\in (-\infty,\infty] \label{ldlambd}
\end{align}
is independent of the choice of $a>0$. We note that $\mathcal{L}(d,\Lambda)$ can be simply expressed in terms of the characteristic couple $(d,\Lambda)$ of $Y$. Indeed, for any $\bm{x}\in\mathcal{Q}$ we set $N(\bm{x}):=\sharp \{ n \geq 1; \ x_n=0\}$ and note that \eqref{condbranchinglplight} implies in particular that $\int_{\mathcal{Q}}(N(\bm{x})-1)_{+}\Lambda(\mathrm{d}\bm{x})<\infty$. It is easy to see from \eqref{laplaceexpobranchingsubtrunc}, monotone convergence and dominated convergence that 
\begin{align}
\mathcal{L}(d,\Lambda)=\infty\un_{\{d>0\}} + \Lambda(\{\bm{x}\in\mathcal{Q}; N(\bm{x})=0\}) - \int_{\mathcal{Q}}(N(\bm{x})-1)_{+}\Lambda(\mathrm{d}\bm{x}), \label{limkappa0}
\end{align}
where $\Lambda(\{\bm{x}\in\mathcal{Q}; N(\bm{x})=0\}) \in [0,\infty]$. 

\begin{remark} \label{casemomexpfini}
If the assumption \eqref{condbranchinglp3} holds true for some $\theta\geq0$ (we recall that in this case the Laplace exponent $\kappa(\cdot)$ introduced in \eqref{laplaceexpobranchingsub} is well-defined on $[\theta,\infty)$), then we have 
\begin{align}
\mathcal{L}(d,\Lambda) = \sup_{\lambda\geq \theta} \kappa(\lambda) = \lim_{\lambda \rightarrow \infty} \kappa(\lambda)=\infty\un_{\{d>0\}} + \Lambda(\{\bm{x}\in\mathcal{Q}; N(\bm{x})=0\}) - \int_{\mathcal{Q}}(N(\bm{x})-1)_{+}\Lambda(\mathrm{d}\bm{x}). \label{limkappa}
\end{align}
\end{remark}
Our criterion for finiteness or infiniteness of $\underline{Y}(\infty)$ is as follows. 
\begin{prop} \label{criteria}
For any $d\geq0$ and for any measure $\Lambda$ on $\mathcal{Q}$ satisfying 
\eqref{condbranchinglp0} and \eqref{condbranchinglplight}, the branching subordinator $Y$ with characteristic couple $(d,\Lambda)$ satisfies $\mathbb{P}(\underline{Y}(\infty)=\infty)\in\{0,1\}$ and 
\begin{itemize}
\item If $\mathcal{L}(d,\Lambda)>0$ then $\mathbb{P}(\underline{Y}(\infty)=\infty)=1$. 
\item If $\mathcal{L}(d,\Lambda)<0$ then $\mathbb{P}(\underline{Y}(\infty)=\infty)=0$. 
\item If $\mathcal{L}(d,\Lambda)=0$ then both $\mathbb{P}(\underline{Y}(\infty)=\infty)=1$ and $\mathbb{P}(\underline{Y}(\infty)=\infty)=0$ are possible. 
\end{itemize}
\end{prop}
%\begin{remark}
%In the case where the assumption \eqref{condbranchinglp3} holds true for some $\theta\geq0$, then the Laplace exponent $\kappa(\cdot)$ defined in \eqref{laplaceexpobranchingsub} is well-defined and the criterion of Proposition \ref{criteria} can be expressed directly in terms of $\kappa(\cdot)$, see Lemma \ref{criteriatruncated} from Section \ref{proofcriteria}. 
%\end{remark}

\begin{exam} \label{applbrpois}
In the example of the branching Poisson process from Section \ref{branchingpoisson}, one has $\mathcal{L}(0,\Lambda^{bp}_{r,\rho})=r-\rho$. We thus get that, in this case, $\underline{Y}(\infty)=\infty$ a.s. if $r>\rho$ while $\underline{Y}(\infty)<\infty$ a.s. if $r<\rho$. In the case $r=\rho$, a direct analysis of the process performed in Section \ref{exbrpoisproclimlaw} shows that $\underline{Y}(\infty)=\infty$ a.s. In the case $r<\rho$, one can even determine the law of $\underline{Y}(\infty)$, which is done in Section \ref{exbrpoisproclimlaw}. 
\end{exam}
The analogue of the event $\{\underline{Y}(\infty)<\infty\}$ is referred to as the phenomenon of \textit{explosion} in \cite{Amini_Devroye_Griffiths_Olver} and \cite{AKMS25}. This terminology comes from the interpretation of the process as an age-dependent branching processes. Focused on branching random walks with non-negative displacements, authors of \cite{Amini_Devroye_Griffiths_Olver} provide conditions for explosion to happen. They also prove a result about the behavior of the leftmost particle in the case of \textit{non-explosion}, that is when $\{\underline{Y}(\infty)=\infty\}$, see the end of section \ref{ClassicalResults} of the present article. In our case, in view of Proposition \ref{criteria}, it is sufficient to assume $\mathcal{L}(d,\Lambda)>0$ to ensure that explosion does not occur. 

\subsubsection{Classical results and linear behavior}\label{ClassicalResults}

A natural question is that of the rate of growth of $\underline{Y}(t)$. In any case, under \eqref{condbranchinglp0} and \eqref{condbranchinglplight}, we have, $\P$-almost surely
\begin{align}\label{linearbehaviour0}
    d\leq\liminf_{t\to\infty}\frac{\underline{Y}(t)}{t}\leq\limsup_{t\to\infty}\frac{\underline{Y}(t)}{t}<\infty.
\end{align}
As for branching random walks (see for example the famous works of Hammersley \cite{Hammersley}, Kingman \cite{Kingman} and Biggins \cite{Biggins}), if \eqref{condbranchinglp3} holds for some $\theta\geq 0$ and if we further require that $-\infty<\kappa(\theta_0)<0$ for some $\theta_0>0$ (allowing $\kappa(0)=-\infty$), then the limit exists and is nonrandom, more precisely, we have $\P$-almost surely 
\begin{align}\label{linearbehaviour}
    \lim_{t\to\infty}\frac{\underline{Y}(t)}{t}=\sup_{\lambda\geq\theta_0}\frac{\kappa(\lambda)}{\lambda}=\sup_{\lambda>0}\frac{\kappa(\lambda)}{\lambda}\in[0,\infty).
\end{align}
Therefore, in the latter case, \eqref{linearbehaviour} provides more information about the behaviors described in Proposition \ref{criteria}: The growth of $t\mapsto\underline{Y}(t)$ is linear if $\mathcal{L}(d,\Lambda)>0$. The growth is sub-linear if $\mathcal{L}(d,\Lambda)=0$, even when $\underline{Y}(\infty)=\infty$; the branching Poisson process in the case $r=\rho$ provides such an example, see Section \ref{brpoiscritcase} for details. Finally, recall that $\underline{Y}(\infty)\in[0,\infty)$ if $\mathcal{L}(d,\Lambda)<0$. \\
Note that, by \eqref{linearbehaviour0}, the classical linear growth is observed when $d>0$, even when the assumptions ensuring \eqref{linearbehaviour} do not hold, for example in the case \eqref{noexpomom}. We discuss the linear behavior and provide a brief justification of \eqref{linearbehaviour0} and \eqref{linearbehaviour} in Appendix \ref{prooflinearbehav}. The justification of \eqref{linearbehaviour} consists in transposing to our setting the classical result obtained for branching random walks in the above mentioned references. 

\vspace{0.2cm}
\subsubsection{Relaxing the exponential moments assumption, sub-linear behavior}

In the present paper, we are mostly interested in branching subordinators with characteristics $(0,\Lambda)$ and $\Lambda$ falling in the case \eqref{noexpomom}. Interestingly, taking $d=0$ and relaxing the exponential moments assumption \eqref{condbranchinglp3} allows to observe sub-linear behaviors for $\underline{Y}(t)$, even when $\mathcal{L}(0,\Lambda)>0$. We will prove that, under some technical conditions on $\Lambda$, $\P$-almost surely
\begin{align}
0 < \liminf_{t \rightarrow \infty} \frac{\underline{Y}(t)}{t^{\gamma}}\leq\limsup_{t \rightarrow \infty} \frac{\underline{Y}(t)}{t^{\gamma}} <\infty, \label{asymptest}
\end{align}
with an explicit $\gamma \in (0,1)$ depending on $\Lambda$. To show this, we need to introduce some assumptions on the truncated Laplace exponent $\kappa_a(\cdot)$. The first assumption is about the term $\phi$ in \eqref{laplaceexpospecialcase}. It says that canonical trajectories (see Definition \ref{DefCanoTraj}) have a stable-like behavior. Note that we often denote $a\sim b$ if $a/b$ converges to $1$. 
\begin{assum} \label{stablelike}
There is $\alpha \in (0,1)$ and $C>0$ such that $\phi(\lambda)\sim C \lambda^{\alpha}$ as $\lambda$ goes to infinity. 
\end{assum}
We note that Assumption \ref{stablelike} is a requirement only on the small jumps of the canonical trajectories but not on the big jumps. 

We define a measure $\mu$ on $[0,\infty)$ by
\begin{align}
\forall A \in \mathcal{B}(\mathbb{R}_+), \ \mu(A):=\int_{\mathcal{Q}} \left ( \sum_{k=2}^{\infty} \mathds{1}_{\{x_k \in A\}} \right ) \Lambda(\mathrm{d}\bm{x}). \label{defmeasmu}
\end{align}
The measure $\mu$ can be understood as a measure-valued average of the contributions of branchings. We note that, for $a>0$, $M_a(\lambda)=\int_{[0,a)}e^{-\lambda\mathfrak{z}}\mu(\mathrm{d}\mathfrak{z})$. If we had $\mu([0,a))$ of order $e^{ca}$, then \eqref{condbranchinglp3} would be satisfied for any $\theta>c$. Since we are mostly interested in the case \eqref{noexpomom}, we assume an higher rate of growth for $\mu([0,a))$, but also some regularity. More precisely, our second assumption is the following. 
\begin{assum} \label{regbranchoriginal}
There is $\sigma>0$ such that 
\begin{align}
0 < \liminf_{a \rightarrow \infty} \frac{\log\big(\mu([0,a))\big)}{a^{1+\sigma}} \leq \limsup_{a \rightarrow \infty} \frac{\log\big(\mu([0,a))\big)}{a^{1+\sigma}} < \infty. \label{assump2meas}
\end{align}
\end{assum}
Since it will be convenient in the proofs to have our assumptions formulated in terms of the truncated Laplace exponent (more precisely, in terms of the functions $\phi(\cdot)$ and $M_{\cdot}(\cdot)$ appearing in its decomposition \eqref{laplaceexpospecialcase}), in Appendix \ref{equivassump} we provide an equivalent formulation of Assumption \ref{regbranchoriginal} in terms of $M_{\cdot}(\cdot)$. We refer to the equivalent formulation as Assumption \ref{regbranch} and the equivalence is proved in Lemma \ref{equivassumption} of Appendix \ref{equivassump}. 

For any $a, \lambda \geq 0$ we set 
\begin{align}
M_{2,a}(\lambda):=\int_{\mathcal{Q}} \left ( \sum_{k=2}^{\infty} e^{-\lambda x_k} \mathds{1}_{\{x_k<a\}} \right )^2 \Lambda(\mathrm{d}\bm{x}). \label{defm2}
\end{align}

Our third assumption can be seen as a control of the variance of the effect of branching events. 
\begin{assum} \label{controlvar}
Assumption \ref{regbranchoriginal} is satisfied and for $\sigma>0$ as in that assumption, there is $c_0>0$ such that for all $c \in (0,c_0]$ we have 
\begin{align}
\limsup_{a \rightarrow \infty} \frac{M_{2,a}(c a^{\sigma})}{(M_a(c a^{\sigma}))^2}<\infty. \label{controlvareq}
\end{align}
\end{assum}
The variance control from Assumption \ref{controlvar} will allow, in Section \ref{lowestpartupperbound}, to show that some functionals involved in estimating the deviation probability for the leftmost particle concentrate around their means, which are themselves expressed in terms of quantities of the form $M_a(c a^{\sigma})$ (for some $c$ and $a$) whose behavior can be inferred from Assumption \ref{regbranchoriginal}. Without Assumption \ref{controlvar}, it might be difficult to rigorously relate the behavior of the system to the quantities of the form $M_a(c a^{\sigma})$. 

\begin{remark} [$p$-ary branchings] \label{binarycase}
By Jensen inequality we have, for $p \geq 2$, $( \sum_{k=2}^p e^{-\lambda x_k} \mathds{1}_{\{x_k<a\}} )^2 \leq (p-1) \sum_{k=2}^p e^{-2\lambda x_k} \mathds{1}_{\{x_k<a\}}$. Therefore, if there is $p \geq 2$ such that the system has at most $p$-ary branchings, that is, if $\Lambda(\{\bm{x}\in\mathcal{Q}; x_{p+1}<\infty\})=0$, then we have $M_{2,a}(\lambda)\leq (p-1)M_{a}(2\lambda)$ for any $a, \lambda \geq 0$ (if the system has only binary branchings then we even have $M_{2,a}(\lambda)=M_{a}(2\lambda)$). This will allow to conveniently check that Assumption \ref{controlvar} is satisfied in some examples involving binary or $p$-ary branchings (see Section \ref{manybinbr} and Remark \ref{trinarycase}). 
\end{remark}

Our last assumption says that a given particle does not jump and give birth at the same time (even though its children can be born with positive displacement), so that jumps and branching events are independent. 
\begin{assum} \label{nojumpatbranchings}
We have $\Lambda(\{\bm{x}\in \mathcal{Q}; 0<x_1 \leq x_2 < \infty\})=0$. 
\end{assum}
The independence, provided by Assumption \ref{nojumpatbranchings}, of jumps and branching events allows for slow particles to not be penalized by the fact that they are slow, so that their branching structure is distributed as the branching structure of a generic particle. 

Now that all the required assumptions have been defined, we can state our main result on the asymptotic behavior of the leftmost particle of a branching subordinator that satisfies those assumptions. 
\begin{theo} \label{generalcriteria}
Let $\Lambda$ be a measure on $\mathcal{Q}$ satisfying \eqref{condbranchinglp0}-\eqref{condbranchinglplight2} and Assumption \ref{stablelike} for some $\alpha \in (0,1)$, Assumption \ref{regbranchoriginal} for some $\sigma>0$, and Assumptions \ref{controlvar} and \ref{nojumpatbranchings}. Then the branching subordinator with characteristic couple $(0,\Lambda)$ satisfies \eqref{asymptest} with $\gamma:=\frac{1}{\alpha + (1-\alpha)(1+\sigma)}$. 
\end{theo}

As a consequence of Theorem \ref{generalcriteria}, we obtain the asymptotic behavior of the leftmost particle in the case of the examples from Section \ref{exampleintro}. 

\begin{cor} \label{corri}
For any $\alpha, \beta \in (0,1)$, the branching subordinator with characteristic couple $(0,\Lambda^{ri}_{\alpha,\beta})$ defined in Section \ref{rareinfbrintro} satisfies \eqref{asymptest} with $\gamma:=\frac{\beta}{\alpha \beta + 1-\alpha}$. 
\end{cor}

\begin{cor} \label{cormb}
For any $\alpha \in (0,1)$ and $\theta>1$, the branching subordinator with characteristic couple $(0,\Lambda^{mb}_{\alpha,\theta})$ defined in Section \ref{manybinbrintro} satisfies \eqref{asymptest} with $\gamma:=\frac{1}{\alpha + (1-\alpha)\theta}$. 
\end{cor}
We will see in Remark \ref{trinarycase} that the result of Corollary \ref{cormb} remains true with the same value of $\gamma$ when the binary branchings from the example of Section \ref{manybinbrintro} are replaced by $p$-ary branchings.

\vspace{0.2cm}

The fact that the growth of $t\mapsto\underline{Y}(t)$ is slower than linear under our assumptions can informally be explained as follows: in cases where $\lim_{a\to\infty}M_a(\lambda)=\infty$ for any $\lambda\geq 0$, each particle in $Y$ is somehow allowed to give progeny to a very large number of children ({see the example from Section \ref{rareinfbrintro} for instance) or/and to create newborn children very often ({see the example from Section \ref{manybinbrintro} for instance), producing a population of particles sufficiently large to maintain $\underline{Y}(t)$ below $\mathfrak{c}t$ for any $\mathfrak{c}>0$. Besides, one can notice that the rate of growth of $t\mapsto\underline{Y}(t)$ {depends on both $\alpha$ (see Assumption \ref{stablelike}) and $\sigma$ (see Assumption \ref{regbranchoriginal}). The parameter $\alpha$ acts on the intensity of the (infinitely many) small jumps of each particle {while $\sigma$ acts on the distribution of the positions of the children of each particle and the combination of both is largely responsible for this rate of growth. 

The idea of the proof of Theorem \ref{generalcriteria} is as follows. Proving the lower bound relies on estimating the average contribution of slow particles. Proving the upper bound is more involved and relies on understanding the mechanism by which slow particles are generated and to show that, in some sense, the structure of branchings allows for the set of slow particles to sustain itself over time. In other words we produce events of large probability on which slow trajectories produce sufficiently enough slow particles. This requires to prove a large deviation estimate that gives a lower bound on the probability for a particle to be slow, and to show that there is some regularity on the distribution of particles that are born from the slow trajectories. 

Properties of branching processes are often established through martingale technics. In the light of this, it seems that another approach might be possible to prove the upper bound in Theorem \ref{generalcriteria}. While the martingale $(W_{\lambda}(t))_{t\geq 0}$ from \eqref{defwlambdat} is not well-defined under the assumptions of Theorem \ref{generalcriteria}, one can define an analogue of that martingale for the truncated process $Y^a$, where the truncation $a$ increases with time. Then, one might be able to proceed as in \cite{Bertoin_Mallein_BiggMart} (using a spinal decomposition argument) to determine suitable conditions for that martingale to have a positive limit. Then, if one can show that the contribution to that martingale of particles on the right of position $c t^{\gamma}$ (for some constant $c>0$) converges to zero almost surely, one can deduce the almost sure existence of particles on the left of $c t^{\gamma}$ at all large times $t$, yielding the result. In comparison, the approach that we used to prove Theorem \ref{generalcriteria} seems more direct, probably more concise, and provides more insights on the mechanism by which slow particles arise and give birth to new slow particles before speeding up.

\vspace{0.2cm}

Although the article \cite{Amini_Devroye_Griffiths_Olver} is mainly devoted to explosion, their Theorem 6.1 is dedicated to the case of non-explosion. Considering a branching random walk $X=(X(n))_{n\geq 1}$ with non-negative and\textrm{ i.i.d }displacements, authors of \cite{Amini_Devroye_Griffiths_Olver} proved a rather precise result: assuming there is non-explosion for the random walk $X$, together with a heavy tail assumption on the offspring distribution, they show that, conditionally on the non-extinction of the system, the renormalized position of the leftmost particle $\underline{X}(n)/x_n$ converges to $1$ almost surely, where $(x_n)_{n\geq 1}$ is a deterministic semi-explicit sequence of positive numbers. We note that authors of \cite{Amini_Devroye_Griffiths_Olver} considered\textrm{ i.i.d }displacements (that assumption plays an important role in the definition of their renormalizing sequence $(x_n)_{n\geq 1}$) which is a setting completely different from the cases we consider.

\vspace{0.3cm}
\subsubsection{Organization of the paper} \label{OrgaPaper}

The rest of the paper is organized as follows. In Section \ref{applonexa}, Theorem \ref{generalcriteria} is applied to the last two examples from Section \ref{exampleintro} and Corollaries \ref{corri} and \ref{cormb} are subsequently derived. In Section \ref{exmainprop} we establish Proposition \ref{criteria} and discuss further the example of the branching Poisson process. Section \ref{sublinbehav} is the bulk of the paper, it contains the proof of Theorem \ref{generalcriteria}. In Appendix \ref{construction} we provide a rigorous construction of branching subordinators and justify some of their basic properties, in Appendix \ref{prooflinearbehav} we justify that \eqref{linearbehaviour0} and \eqref{linearbehaviour} can be derived from classical results on branching random walks, and in Appendix \ref{equivassump} we study an equivalent formulation of Assumption \ref{regbranchoriginal}.

\vspace{0.3cm}

\section{Application of Theorem \ref{generalcriteria} on two examples} \label{applonexa}

In this section, we prove that the last two examples from Section \ref{exampleintro} satisfy the assumptions of Theorem \ref{generalcriteria} in order to prove Corollaries \ref{corri} and \ref{cormb}. Even though those assumptions may seem a little abstract, we note that, at least in our two examples, they are rather straightforward to check. 

\subsection{$\alpha$-stable trajectories with rare infinite branchings: proof of Corollary \ref{corri}} \label{rareinfbr}

We denote by $\kappa^{ri,\alpha,\beta}_{\cdot}(\cdot)$ the truncated Laplace exponent associated with $(0,\Lambda^{ri}_{\alpha,\beta})$. The decomposition \eqref{laplaceexpospecialcase}-\eqref{laplaceexpospecialcase2} of the truncated Laplace exponent $\kappa^{ri,\alpha,\beta}_{\cdot}(\cdot)$ reads $\kappa^{ri,\alpha,\beta}_a(\lambda)=\phi^{\alpha}(\lambda) -M^{ri,\beta}_a(\lambda)$ where $\phi^{\alpha}(\lambda):=\lambda^{\alpha}$ and $M^{ri,\beta}_a(\lambda):=\sum_{k=2}^{\infty} e^{-\lambda (\log k)^{\beta}} \mathds{1}_{\{(\log k)^{\beta}<a\}}$. In particular, Assumption \ref{stablelike} is satisfied in this example. Let $\mu_{ri,\beta}$ be the measure defined by taking $\Lambda=\Lambda^{ri}_{\alpha,\beta}$ in \eqref{defmeasmu}. We have $\mu_{ri,\beta}([0,a))=\sharp \{ k \geq 2; (\log k)^{\beta}<a \}=\lfloor e^{a^{1/\beta}} \rfloor-1$. This shows that Assumption \ref{regbranchoriginal} is satisfied with $\sigma:=(1-\beta)/\beta$. 
We denote by $M^{ri,\beta}_{2,\cdot}(\cdot)$ the function $M_{2,\cdot}(\cdot)$ associated to $\Lambda^{ri}_{\alpha,\beta}$ via \eqref{defm2}. We have clearly $M^{ri,\beta}_{2,a}(\lambda)=(M^{ri,\beta}_{a}(\lambda))^2$ for any $a, \lambda \geq 0$ so, in the present example, Assumption \ref{controlvar} is satisfied. Finally, $\Lambda^{ri}_{\alpha,\beta}$ clearly satisfies Assumption \ref{nojumpatbranchings}. Since all assumptions are satisfied, we can apply Theorem \ref{generalcriteria} and get Corollary \ref{corri}. 

\subsection{$\alpha$-stable trajectories with many binary branchings: proof of Corollary \ref{cormb}} \label{manybinbr}

We denote by $\kappa^{mb, \alpha,\theta}_{\cdot}(\cdot)$ the truncated Laplace exponent associated with $(0,\Lambda^{mb}_{\alpha,\theta})$. The decomposition \eqref{laplaceexpospecialcase}-\eqref{laplaceexpospecialcase2} of the truncated Laplace exponent $\kappa^{mb, \alpha,\theta}_{\cdot}(\cdot)$ reads $\kappa^{mb, \alpha,\theta}_a(\lambda)=\phi^{\alpha}(\lambda) -M^{mb, \theta}_a(\lambda)$ where $\phi^{\alpha}(\lambda):=\lambda^{\alpha}$ and $M^{mb, \theta}_a(\lambda):=\int_0^a e^{\mathfrak{z}^{\theta}-\mathfrak{z}\lambda}\mathrm{d}\mathfrak{z}$. In particular, Assumption \ref{stablelike} is satisfied in this example. Let $\mu_{mb,\theta}$ be the measure defined by taking $\Lambda=\Lambda^{mb}_{\alpha,\theta}$ in \eqref{defmeasmu}. We have $\mu_{mb,\theta}([0,a))=\int_0^a e^{\mathfrak{z}^{\theta}}\mathrm{d}\mathfrak{z}$ so $a e^{(a/2)^{\theta}}/2 \leq \mu_{mb,\theta}([0,a)) \leq a e^{a^{\theta}}$. This shows that Assumption \ref{regbranchoriginal} is satisfied with $\sigma:=\theta-1$. We now prove the following lemma. 
\begin{lemma}\label{upperboundexponentmb}
For any $\ell>0$ and $a> 0$ we have $M^{mb,\theta}_{a}(\ell a^{\theta-1}) \leq a^{1-\theta} e^{a^{\theta}}$. For any $\ell \in (0,1/3)$ and $a> 0$ we have $M^{mb,\theta}_{a}(\ell a^{\theta-1}) \geq a^{1-\theta} (e^{2a^{\theta}/3}-e^{2\delta a^{\theta}/3})$ where $\delta:=(\ell+2/3)^{1/(\theta-1)}$. 
\end{lemma}

\begin{proof}
%For the first point we assume that $\lambda \geq 2u^{\theta-1}$. We have 
%\begin{align*}
%M^{mb,\theta}_{u}(\lambda) = \int_0^{u} e^{z^{\theta}-\lambda z} dz \leq \int_{0}^{u} e^{-\lambda z/2}dz \leq \int_{0}^{\infty} e^{-\lambda z/2}dz = \frac{2}{\lambda} \leq u^{1-\theta}, 
%\end{align*}
%where we have used that $z^{\theta}-\lambda z\leq -\lambda z/2$ for all $z \in [0,u]$ when $\lambda \geq 2u^{\theta-1}$. The first point follows. 
We fix $\ell>0$ and set $\lambda(a):=\ell a^{\theta-1}$. We have, 
%\begin{align*}
%M^{mb,\theta}_{u}(\lambda) = \int_0^{u} e^{z^{\theta}-\lambda z} dz \geq \int_{u/2}^{u} e^{\lambda z}dz \geq \frac{u}{2} e^{\lambda u/2}, 
%\end{align*}
\begin{align*}
M^{mb,\theta}_{a}(\lambda(a)) = \int_0^{a} e^{\mathfrak{z}^{\theta}-\mathfrak{z}\lambda(a)}\mathrm{d}\mathfrak{z} \leq \int_{0}^{a} e^{\mathfrak{z}\lambda(a)/\ell}\mathrm{d}\mathfrak{z} = \frac{e^{\lambda(a) a/\ell}-1}{\lambda(a)/\ell} \leq a^{1-\theta} e^{a^{\theta}}, 
\end{align*}
where we have used that $\mathfrak{z}^{\theta}-\mathfrak{z}\lambda(a)\leq \mathfrak{z}\lambda(a)/\ell$ for all $\mathfrak{z}\in [0,a]$ when $\lambda(a)=\ell a^{\theta-1}$. The first point follows. We now consider $\ell\in(0,1/3)$ and $\lambda(a):=\ell a^{\theta-1}$. We set $\delta:=(\ell+2/3)^{1/(\theta-1)} \in (0,1)$. Then 
\begin{align*}
M^{mb,\theta}_{a}(\lambda(a)) \geq \int_{\delta a}^{a} e^{2\mathfrak{z}\lambda(a)/3\ell}\mathrm{d}\mathfrak{z} = \frac{e^{2\lambda(a) a/3\ell}-e^{2\delta\lambda(a) a/3\ell}}{2\lambda(a)/3\ell} = \frac{3}{2} a^{1-\theta} (e^{2a^{\theta}/3}-e^{2\delta a^{\theta}/3}), 
\end{align*}
where we have used that $\mathfrak{z}^{\theta}-\mathfrak{z}\lambda(a)\geq 2\mathfrak{z}\lambda(a)/3\ell$ for all $\mathfrak{z}\in [\delta a,a]$ when $\lambda(a)=\ell a^{\theta-1}$. The second point follows. 
\end{proof}
We denote by $M^{mb,\theta}_{2,\cdot}(\cdot)$ the function $M_{2,\cdot}(\cdot)$ associated to $\Lambda^{mb}_{\alpha,\theta}$ via \eqref{defm2}. By Remark \ref{binarycase}, we have $M^{mb,\theta}_{2,a}(\lambda)=M^{mb,\theta}_{a}(2\lambda)$ for any $a, \lambda \geq 0$. Therefore, taking any $c\in(0,1/3)$ and applying Lemma \ref{upperboundexponentmb} (the first point with $\ell=2c$ and the second point with $\ell=c$) we get \eqref{controlvareq} so, in the present example, Assumption \ref{controlvar} is satisfied. Finally, $\Lambda^{mb}_{\alpha,\theta}$ clearly satisfies Assumption \ref{nojumpatbranchings}. Since all assumptions are satisfied, we can apply Theorem \ref{generalcriteria} and get Corollary \ref{cormb}. 

\begin{remark} [$p$-ary branchings] \label{trinarycase}
In the present example, let us set $p\geq 3$ and replace binary branchings with $p$-ary branchings by setting that, at each branching, $p-1$ particles are born (with identical displacements) instead of one. Let us denote by $\Lambda^{mp}_{\alpha,\theta}$ the corresponding jump-branching measure and by $\kappa^{mp, \alpha,\theta}_{\cdot}(\cdot)$ the truncated Laplace exponent of the branching subordinator with characteristic couple $(0,\Lambda^{mp}_{\alpha,\theta})$. The decomposition \eqref{laplaceexpospecialcase}-\eqref{laplaceexpospecialcase2} of $\kappa^{mp, \alpha,\theta}_{\cdot}(\cdot)$ reads $\kappa^{mp, \alpha,\theta}_a(\lambda)=\phi^{\alpha}(\lambda) -M^{mp, \theta}_a(\lambda)$ where $\phi^{\alpha}$ is as before and $M^{mp, \theta}_a(\lambda)=(p-1)M^{mb, \theta}_a(\lambda)$. Moreover, $M^{mp,\theta}_{2,a}(\lambda)=(p-1)^2 M^{mb,\theta}_{a}(2\lambda)$. We thus see that all the assumptions of Theorem \ref{generalcriteria} are identically satisfied in this case so Corollary \ref{cormb} still holds in this case with the same value of $\gamma$. 
\end{remark}

\section{Finite or infinite limit for the leftmost particle} \label{exmainprop}

\subsection{Proof of Proposition \ref{criteria}} \label{proofcriteria}

In this subsection we prove Proposition \ref{criteria}. For this we consider $d\geq0$ and a measure $\Lambda$ on $\mathcal{Q}$ satisfying \eqref{condbranchinglp0}-\eqref{condbranchinglplight}, and we consider the branching subordinator $Y$ with characteristic couple $(d,\Lambda)$. We let $\kappa_{\cdot}(\cdot)$ denote its truncated Laplace exponent defined by \eqref{laplaceexpobranchingsubtrunc}. The first part of the proposition is proved in the following easy lemma. 
\begin{lemma} \label{01law}
We have $\mathbb{P}(\underline{Y}(\infty)=\infty)\in\{0,1\}$. 
\end{lemma}
To prove Lemma \ref{01law} we use that the property of the leftmost particle converging to $\infty$ is \textit{inherited} in the sense that the property is also satisfied by all subsystems formed by one particle and its future lineage. It is a classical fact that, for Galton-Watson trees, any inherited property occurs with probability $0$ or $1$, conditionally on non-extinction, see Proposition 5.6 in \cite{LyonsPeres2017}. The following proof adapts this simple idea to branching subordinators. 
\begin{proof}[Proof of Lemma \ref{01law}]
Recall that $\mathcal{Q}^{\star}=\{\bm{x}\in\mathcal{Q}; x_2 < \infty\}$. Since $\Lambda(\mathcal{Q}^{\star})>0$ (see \eqref{condbranchinglp0}) and $\Lambda(\{\bm{x}\in\mathcal{Q};\; x_2< a\})<\infty$ for any $a>0$ (see \eqref{condbranchinglplight}), one can choose $a_1>0$ such that $\Lambda(\{\bm{x}\in\mathcal{Q};\; x_2< a_1\})\in(0,\infty)$. Denote by $\mathcal{Q}'$ the set $\{\bm{x}\in\mathcal{Q};\; x_2< a_1\}$. 
Let $T$ be the smallest time where a branching in $\mathcal{Q}'$ occurs for the initial particle $\varnothing$; we denote this branching by $\bm{x}(T)=(x_n(T))_{n\geq 1}$. Clearly, $T$ is positive and finite (by the construction from Appendix \ref{construction}, it follows an exponential distribution with parameter $\Lambda(\mathcal{Q}')$) and is a stopping time. Since $\bm{x}(T) \in \mathcal{Q}^{\star}$ we have $x_1(T)\leq x_2(T)<\infty$ and, at $T$, the system contains at least two particles, one being at position $Z_1:=Y_{\varnothing}(T-)+x_1(T)$ and one at position $Z_2:=Y_{\varnothing}(T-)+x_2(T)$. We denote by $Y_1$ and $Y_2$ the systems issued from these two particles, with time shifted by $T$ and positions shifted down by $Z_1$ and $Z_2$ respectively. By the strong branching property at time $T$ (see Remark \ref{ReguBS}), $Y_1$ and $Y_2$ are independent branching subordinators distributed as $Y$. We have $\underline{Y}(\infty) \leq (Z_1+\underline{Y_1}(\infty)) \wedge (Z_2+\underline{Y_2}(\infty))$ so $\{\underline{Y}(\infty)=\infty\} \subset \{\underline{Y_1}(\infty)=\infty\} \cap \{\underline{Y_2}(\infty)=\infty\}$. We thus get $\mathbb{P}(\underline{Y}(\infty)=\infty)\leq \mathbb{P}(\underline{Y}(\infty)=\infty)^2$, which yields the result. 
\end{proof}

Our second step in the proof of Proposition \ref{criteria} is to compare $\underline{Y}(\infty)$ and $\underline{Y}^a(\infty)$. We recall that, for any $a>a_0(\Lambda)$ (see Section \ref{moredef}), $Y^a$ denotes the $a$-truncation of $Y$, described in Section \ref{moredef} and rigorously introduced in Definition \ref{TroncBranchingSubor}. As can be seen from \eqref{linearbehaviour} and \eqref{asymptest}, the asymptotic behaviors of  $\underline{Y}^a(t)$ and $\underline{Y}(t)$ can be very different in the case \eqref{noexpomom}, thus making possibly delicate the comparison between $\underline{Y}(\infty)$ and $\underline{Y}^a(\infty)$. The following lemma however shows that the finiteness of $\underline{Y}(\infty)$ and $\underline{Y}^a(\infty)$ are equivalent for large $a$. 
\begin{lemma} \label{truncation}
There exists $\Tilde{a}>a_0(\Lambda)$ such that for any $a\geq\Tilde{a}$ we have $\mathbb{P}(\underline{Y}^a(\infty)=\infty)=\mathbb{P}(\underline{Y}(\infty)=\infty)$. 
\end{lemma}

\begin{proof}
Let us fix $a>a_0(\Lambda)$. In the natural coupling between $Y^a$ and $Y$, $Y^a$ is a subsystem of $Y$ so we have almost surely $\underline{Y}(t)\leq\underline{Y}^a(t)$ for all $t\geq 0$. In particular we have $\underline{Y}(\infty)\leq\underline{Y}^a(\infty)$ almost surely so $\mathbb{P}(\underline{Y}(\infty)=\infty)\leq\mathbb{P}(\underline{Y}^a(\infty)=\infty)$. Therefore, if $\mathbb{P}(\underline{Y}(\infty)=\infty)=1$, the claim holds with, for example, $\Tilde{a}=a_0(\Lambda)+1$. 

By Lemma \ref{01law}, we can now assume that $\mathbb{P}(\underline{Y}(\infty)=\infty)=0$. Then, there is $\Tilde{a}>a_0(\Lambda)$ such that $\mathbb{P}(\underline{Y}(\infty)<\Tilde{a})>0$. We now fix any $a\geq\Tilde{a}$. We see from the construction of Appendix \ref{construction} that $t\mapsto\underline{Y}(t)$ is almost surely non-decreasing on $(0,\infty) \cap \mathbb{Q}$. We thus have, almost surely on $\{\underline{Y}(\infty)<a\}$, that $\underline{Y}(t)<a$ for all $t \in (0,\infty) \cap \mathbb{Q}$. Therefore, for each $t \in (0,\infty) \cap \mathbb{Q}$, there exists $u(t) \in \mathcal{N}(t)$ such that $Y_{u(t)}(t)<a$ (see the Introduction and Definition \ref{DefBranchingSubor}). Since $s \mapsto Y_{u(t)}(s)$ is almost surely non-decreasing on $[0,t]$ we have almost surely $Y_{u(t)}(s)-Y_{u(t)}(s-) \in [0,a)$ for all $s\in[0,t]$ so, in particular, $u(t) \in \mathcal{N}^a(t)$, where $\mathcal{N}^a(t)$ is the set of particles present in the subsystem $Y^a$ at time $t$ (see Section \ref{moredef} and Definition \ref{TroncBranchingSubor}). We thus have almost surely on $\{\underline{Y}(\infty)<a\}$ that $\underline{Y}^a(t)<a$ for all $t \in (0,\infty) \cap \mathbb{Q}$ and therefore $\underline{Y}^a(\infty)\leq a<\infty$. In particular, $\mathbb{P}(\underline{Y}^a(\infty)<\infty)\geq\mathbb{P}(\underline{Y}(\infty)<a)>0$ so, by Lemma \ref{01law}, we get $\mathbb{P}(\underline{Y}^a(\infty)=\infty)=0$.
\end{proof}

Since the branching subordinator $Y^a$ satisfies the exponential moment condition \eqref{condbranchinglp3} (see Section \ref{moredef}), we are reduced to prove the criterion under that condition. This is done in the following lemma. 
\begin{lemma} \label{criteriatruncated}
Let $d\geq0$ and $\Lambda$ be a measure on $\mathcal{Q}$ satisfying \eqref{condbranchinglp0} and \eqref{condbranchinglp3} (for some $\theta\geq0$), and let $\kappa(\cdot)$ denote its Laplace exponent defined in \eqref{laplaceexpobranchingsub}. Then the branching subordinator $Y$ with characteristic couple $(d,\Lambda)$ satisfies the following: 
\begin{itemize}
\item If $\sup_{\lambda \geq \theta} \kappa(\lambda)>0$ then $\mathbb{P}(\underline{Y}(\infty)=\infty)=1$. 
\item If $\sup_{\lambda \geq \theta} \kappa(\lambda)<0$ then $\mathbb{P}(\underline{Y}(\infty)=\infty)=0$. 
\end{itemize}
\end{lemma}

\begin{proof}
First assume that $\sup_{\lambda\geq \theta} \kappa(\lambda)>0$. We can thus fix $\lambda \geq \theta$ such that $\kappa(\lambda)>0$. Then, using \eqref{defwlambdat}, 
\begin{align*}
e^{-\lambda \underline{Y}(t)} \leq \sum_{u\in\mathcal{N}(t)}e^{-\lambda Y_u(t)} = e^{-t\kappa(\lambda)} W_{\lambda}(t). 
\end{align*}
Combining with the almost sure convergence of the non-negative martingale $(W_{\lambda}(t))_{t\geq 0}$ we get that, almost surely, $e^{-\lambda \underline{Y}(t)}$ converges to $0$ as $t\to\infty$, so $\mathbb{P}(\underline{Y}(\infty)=\infty)=1$. \\

Now assume that $\sup_{\lambda\geq \theta} \kappa(\lambda)<0$. Combining this assumption with \eqref{limkappa} and recalling, from the discussion before Remark \ref{casemomexpfini}, that $\int_{\mathcal{Q}}(N(\bm{x})-1)_{+}\Lambda(\mathrm{d}\bm{x})<\infty$, we get 
\begin{align}
d=0, \ \Lambda(\{\bm{x}\in\mathcal{Q}; N(\bm{x})=0\}) < \int_{\mathcal{Q}}(N(\bm{x})-1)_{+}\Lambda(\mathrm{d}\bm{x})<\infty. \label{dqn}
\end{align}
Let $C_t:=\sum_{u\in\mathcal{N}(t)}\mathds{1}_{\{Y_u(t)=0\}}$. At time $t$, the jump rate of the process $C$ is $C_t \Lambda(\{\bm{x}\in\mathcal{Q}; N(\bm{x})\neq 1\})$ (which is finite by \eqref{dqn}) and the size of jumps has the law of $N(\bm{x})-1$ under $\Lambda(\{\bm{x}\in\mathcal{Q}; N(\bm{x})\neq 1\} \cap \cdot)/\Lambda(\{\bm{x}\in\mathcal{Q}; N(\bm{x})\neq 1\})$. We have $C_0=1$ and $(C_t)_{t\geq 0}$ remains constant equal to $0$ after reaching $0$ for the first time. Let $(D_t)_{t\geq 0}$ be a compound Poisson process starting from $1$, that jumps at rate $\Lambda(\{\bm{x}\in\mathcal{Q}; N(\bm{x})\neq 1\})$ and whose jumps have the same law as those of $(C_t)_{t\geq 0}$. Let $I_t:=\int_0^t(1/D_s)ds$ with the convention $1/0=\infty$. Let $A_t:=\inf \{ s \geq 0; \ I_s > t \}$. A classical argument shows that $(C_t)_{t\geq 0}$ is equal in law to $(D_{A_t})_{t\geq 0}$. Therefore, 
\begin{align}
\mathbb{P}(\underline{Y}(\infty)=0)=\mathbb{P}\left (\forall t \geq 0, C_t \neq 0 \right )=\mathbb{P}\left (\forall t \geq 0, D_t \neq 0 \right ). \label{survat0}
\end{align}
Then, note that 
\begin{align*}
\mathbb{E}[D_1]=\int_{\mathcal{Q}}(N(\bm{x})-1)\Lambda(\mathrm{d}\bm{x})=\int_{\mathcal{Q}}(N(\bm{x})-1)_{+}\Lambda(\mathrm{d}\bm{x})-\Lambda(\{\bm{x}\in\mathcal{Q}; N(\bm{x})=0\})>0, 
\end{align*}
where the positivity comes from \eqref{dqn}. In particular, the L\'evy process $(D_t)_{t\geq 0}$ drifts to infinity so $\mathbb{P}(\forall t \geq 0, D_t \neq 0)>0$. By \eqref{survat0} we get $\mathbb{P}(\underline{Y}(\infty)=0)>0$ so, in particular, $\mathbb{P}(\underline{Y}(\infty)=\infty)<1$. Combining with Lemma \ref{01law} we deduce that $\mathbb{P}(\underline{Y}(\infty)=\infty)=0$. 
\end{proof}
\begin{remark}
The proof of the first point of Lemma \ref{criteriatruncated} relies on the existence of the Laplace exponent $\kappa(\cdot)$ and therefore requires Assumption \eqref{condbranchinglp3}. However, using \eqref{limkappa0} instead of \eqref{limkappa}, the proof of the second point applies under the less restrictive assumptions of Proposition \ref{criteria} and allows to yields directly the second point of that proposition. 
\end{remark}

\begin{remark}
Under the extra assumption that $-\infty<\kappa(\theta_0)<0$ for some $\theta_0>0$, the first point of Lemma \ref{criteriatruncated} is also a consequence of \eqref{linearbehaviour} since in this case $\underline{y}>0$. 
\end{remark}

We can now conclude the proof of Proposition \ref{criteria}. 

\begin{proof}[Proof of Proposition \ref{criteria}]
We assume that we are in the setting of the proposition. The first claim is proved by Lemma \ref{01law}. Let $\Tilde{a}>a_0(\Lambda)$ be as in Lemma \ref{truncation}. We recall that the quantity $\mathcal{L}(d,\Lambda)$, defined in \eqref{ldlambd}, is independent of the choice of $a>0$ so, in particular, $\mathcal{L}(d,\Lambda)=\lim_{\lambda \rightarrow \infty} \kappa_{\Tilde{a}}(\lambda)$. Note that the $\Tilde{a}$-truncation of $Y$, $Y^{\Tilde{a}}$, satisfies the requirements of Lemma \ref{criteriatruncated} and recall from the discussion after \eqref{laplaceexpospecialcase2} that its Laplace exponent is $\kappa_{\Tilde{a}}(\cdot)$. If $\mathcal{L}(d,\Lambda)>0$ (resp. $<0$) we have $\lim_{\lambda \rightarrow \infty} \kappa_{\Tilde{a}}(\lambda)>0$ (resp. $<0$) so, by Lemma \ref{criteriatruncated}, $\mathbb{P}(\underline{Y}^{\Tilde{a}}(\infty)=\infty)=1$ (resp. $=0$), and by Lemma \ref{truncation}, $\mathbb{P}(\underline{Y}(\infty)=\infty)=1$ (resp. $=0$). 

We now provide two examples that both fall in the case $\mathcal{L}(d,\Lambda)=0$ and for which $\mathbb{P}(\underline{Y}(\infty)=\infty)=0$ and $\mathbb{P}(\underline{Y}(\infty)=\infty)=1$ respectively. First consider a branching subordinator $Y$ with characteristic couple $(d,\Lambda)$ as follows: $d=0$ and $\Lambda$ satisfies \eqref{condbranchinglp0}, \eqref{condbranchinglp3}, $\Lambda(\mathcal{Q})\in(0,\infty)$, and $N(\bm{x})=1$ for $\Lambda$-almost every $\bm{x}\in\mathcal{Q}$. Then we see from Remark \ref{casemomexpfini} and \eqref{limkappa} that for this $Y$ we have $\mathcal{L}(d,\Lambda)=0$. Moreover, applying the construction from Appendix \ref{construction} to the current example we see that $Y_{\varnothing}(t)=0$ for all $t \geq 0$ so $\mathbb{P}(\underline{Y}(\infty)=0)=1$ and in particular $\mathbb{P}(\underline{Y}(\infty)=\infty)=0$. The second example is the branching Poisson process from Section \ref{branchingpoisson} in the critical case $r=\rho$. We investigate this example in the following Section \ref{exbrpoisproclimlaw} and show that it satisfies $\mathbb{P}(\underline{Y}(\infty)=\infty)=1$. 
\end{proof}

\subsection{A closer look to the example of the branching Poisson process} \label{exbrpoisproclimlaw}

In this subsection we consider the example of a branching Poisson process $Y$ with parameters $r,\rho>0$, as defined in Section \ref{branchingpoisson}. We recall from Example \ref{applbrpois} that, in this case, $\underline{Y}(\infty)<\infty$ a.s. if $r<\rho$ and $\underline{Y}(\infty)=\infty$ a.s. if $r>\rho$. We first assume $r=\rho$ and show that in this case we also have $\underline{Y}(\infty)=\infty$ a.s. Then we assume $r<\rho$ and characterize the law of $\underline{Y}(\infty)$. 

\subsubsection{The case $r=\rho$} \label{brpoiscritcase}

Without loss of generality we assume that $r=\rho=1$. This $Y$ has characteristic couple $(0,\Lambda^{bp}_{1,1})$ and we see from Example \ref{applbrpois} that $\mathcal{L}(0,\Lambda^{bp}_{1,1})=0$. Defining $(C_t)_{t\geq 0}$ and $(D_t)_{t\geq 0}$ as in the proof of Lemma \ref{criteriatruncated} we have that \eqref{survat0} still holds true but this time $\mathbb{E}[D_1]=\int_{\mathcal{Q}}(N(\bm{x})-1)\Lambda(\mathrm{d}\bm{x})=0$. Since the L\'evy process $(D_t)_{t\geq 0}$ is almost surely non-constant we get that $(D_t)_{t\geq 0}$ oscillates, and since that process only takes integer values and only has jumps of size $1$ and $-1$, we get that $\mathbb{P}(\forall t \geq 0, D_t \neq 0)=0$. Therefore, by \eqref{survat0} we get $\mathbb{P}(\underline{Y}(\infty)=0)=0$. Now let us prove by induction on $k$ that $\mathbb{P}(\underline{Y}(\infty)=k)=0$ for all $k \geq 0$. The claim has been proved for $k=0$ so we now fix $k\geq 0$ and assume that the claim has been proved for all $j \in \{0,\cdots,k\}$. Let $T:=\inf\{t\geq 0; \ \underline{Y}(t)=k+1 \}$. This $T$ is a stopping time and, by the induction hypothesis, we have almost surely $T<\infty$. Let $M:=\sharp \{ u\in\mathcal{N}(T); Y_u(T)=k+1\}$ and let us denote by $Y_1,\cdots,Y_M$ the systems issued from the $M$ particles that are located at $k+1$ at time $T$, shifted down by $k+1$. By the strong branching property at time $T$ (see Remark \ref{ReguBS}), we have that, for any $m \geq 1$, conditionally on $\{M=m\}$, $Y_1,\cdots,Y_m$ are $m$ independent branching subordinators distributed as $Y$. In particular, 
\begin{align*}
\mathbb{P}(\underline{Y}(\infty)=k+1)&=\mathbb{P}(\underline{Y_1}(\infty) \wedge \cdots \wedge \underline{Y_M}(\infty)=0)\\
&=\sum_{m \geq 1} \mathbb{P}(M=m) \times \left ( 1 - (1-\mathbb{P}(\underline{Y}(\infty)=0))^m \right )=0, 
\end{align*}
where we have used that $\mathbb{P}(\underline{Y}(\infty)=0)=0$ in the last equality. This concludes the proof by induction that $\mathbb{P}(\underline{Y}(\infty)=k)=0$ for all $k \geq 0$ so in particular $\mathbb{P}(\underline{Y}(\infty)=\infty)=1$.

We thus have almost sure convergence of $\underline{Y}(t)$ to infinity as $t$ goes to infinity, but we see from \eqref{linearbehaviour} that $\underline{Y}(t)/t$ converges to $0$. The branching Poisson process in the case $r=\rho$ thus provides another instance of sub-linear growth of $\underline{Y}(t)$, even if it does not satisfy the assumptions from Theorem \ref{generalcriteria}. 

\subsubsection{The case $r<\rho$}

%Decomposing on the first transition of the process and using the branching property we get 
%\begin{align*}
%\mathbb{P}(\underline{Y}(\infty)=0) & = \frac{r}{r+\rho} \times 0 + \frac{\rho}{r+\rho} \times (1-(1-\mathbb{P}(\underline{Y}(\infty)=0))^2) \\
%& =\frac{\rho}{r+\rho} \times \mathbb{P}(\underline{Y}(\infty)=0) \times (2-\mathbb{P}(\underline{Y}(\infty)=0)). 
%\end{align*}
%We have from the proof of Lemma \ref{criteriatruncated} that $\mathbb{P}(\underline{Y}(\infty)=0)>0$ so, dividing by $\mathbb{P}(\underline{Y}(\infty)=0)$ in the above equality we get 
%\begin{align*}
%\mathbb{P}(\underline{Y}(\infty)=0) = 1-\frac{r}{\rho}. 
%\end{align*}
We assume $r<\rho$. Let us fix $k\geq 0$ and note from the proof of Lemma \ref{criteriatruncated} that $\mathbb{P}(\underline{Y}(\infty)=0)>0$ so $\mathbb{P}(\underline{Y}(\infty)\geq k+1)<1$. Decomposing on the first transition of the system (we recall that the initial configuration is a single particle located at the position $0$) and using the branching property we get 
\begin{align*}
\mathbb{P}(\underline{Y}(\infty)\geq k+1) & = \frac{r}{r+\rho} \times \mathbb{P}(\underline{Y}(\infty)\geq k) + \frac{\rho}{r+\rho} \times \mathbb{P}(\underline{Y}(\infty)\geq k+1)^2. 
\end{align*}
Solving this quadratic equation for $\mathbb{P}(\underline{Y}(\infty)\geq k+1)$ and noticing that the largest of the two solutions is larger or equal to $1$ we obtain 
\begin{align}
\mathbb{P}(\underline{Y}(\infty)\geq k+1) = \frac{r+\rho-\sqrt{(r+\rho)^2-4r\rho \mathbb{P}(\underline{Y}(\infty)\geq k)}}{2 \rho}. \label{exprrecloi}
\end{align}
Since $\mathbb{P}(\underline{Y}(\infty)\geq 0)=1$, the recursive relation \eqref{exprrecloi} completely characterizes the law of $\underline{Y}(\infty)$ in this example. 

\begin{remark} \label{nosimpleexpr}
If the binary branchings from the branching Poisson process are replaced by $p$-ary branchings for some $p \geq 3$, then obtaining a similar recursive relation for the law of $\underline{Y}(\infty)$ requires to solve a polynomial equation of degree $p$. This suggests that, in general, there is no simple expression for the law of $\underline{Y}(\infty)$, even in the case of branching subordinators supported on $\mathbb{Z}_+$. 
\end{remark}

\section{Sub-linear behavior: proof of Theorem \ref{generalcriteria}} \label{sublinbehav}

\subsection{Lower bound} \label{lowestpartlowerbound}

In this section we assume that all assumptions from Theorem \ref{generalcriteria} hold true and prove 
\begin{align}
\mathbb{P}-a.s. \ \liminf_{t \rightarrow \infty} \frac{\underline{Y}(t)}{t^{\gamma}} >0. \label{minoliminf}
\end{align}

By Lemma \ref{equivassumption} of Appendix \ref{equivassump}, Assumption \ref{regbranchoriginal} is equivalent to Assumption \ref{regbranch} so we can assume that the latter holds true.
The main ingredient of the proof is the following lemma. 

\begin{lemma} \label{boundprobalow}
Let $\alpha, C$ be as in Assumption \ref{stablelike} and $\sigma, c_1$ be as in Assumption \ref{regbranch}. Set $\gamma$ as in Theorem \ref{generalcriteria} and $m_0:=(C/c_1^{1-\alpha})^{\gamma}$. For any $m \in (0,m_0)$ we have 
\begin{align*}
\limsup_{t \rightarrow \infty} \frac{\log \left ( \mathbb{P} \left ( \underline{Y}(t) < m t^{\gamma} \right ) \right )}{t^{\gamma(1+\sigma)}} <0. 
\end{align*}
\end{lemma}

\begin{proof}
We recall that, for any $a>a_0(\Lambda)$ (see Section \ref{moredef}), $Y^a$ denotes the $a$-truncation of $Y$, described in Section \ref{moredef} and rigorously introduced in Definition \ref{TroncBranchingSubor}. 
%In the natural coupling between $Y^u$ and $Y$ (see Section \ref{moredef}), $Y^u$ is a subsystem of $Y$ so we have almost surely $\underline{Y}(s)\leq\underline{Y}^u(s)$ for all $s\geq 0$. 
Let us fix $m \in (0,m_0)$ and $t>(a_0(\Lambda)/m)^{1/\gamma}$. 
%The system $\tilde Y$ has a finite number of particles at any time so we denote by $\tilde N(s)$ the number of particles in that system at any instant $s\geq 0$. 
Reasoning similarly as in the proof of Lemma \ref{truncation}, we get $\{\underline{Y}(t) < m t^{\gamma}\}=\{\underline{Y}^{m t^{\gamma}}(t) < m t^{\gamma}\}$ (where $Y^{m t^{\gamma}}$ is $Y^a$ with the choice $a=m t^{\gamma}$) so, for any $\lambda \geq 0$, we have 
\begin{align*}
\mathbb{P} \left ( \underline{Y}(t) < m t^{\gamma} \right ) & = \mathbb{P} \left ( \underline{Y}^{m t^{\gamma}}(t) < m t^{\gamma} \right )\leq\E\left[\sum_{u\in\mathcal{N}^{mt^{\gamma}}(t)}\un_{\{Y_u^{^{m t^{\gamma}}}(t)<mt^{\gamma}\}}\right] \leq e^{\lambda m t^{\gamma}}\E\left[\sum_{u\in\mathcal{N}^{mt^{\gamma}}(t)}e^{-\lambda Y_u^{^{m t^{\gamma}}}(t)}\right],  
\end{align*}
which is nothing but $e^{\lambda m t^{\gamma} - t \kappa_{m t^{\gamma}}(\lambda)}$ by \eqref{ExpEspTruncature}. Combining the above with \eqref{laplaceexpospecialcase} we get 
\begin{align}
\log \left ( \mathbb{P} \left ( \underline{Y}(t) < m t^{\gamma} \right ) \right ) \leq \lambda m t^{\gamma} - t \lambda^{\alpha} (\lambda^{-\alpha} \phi(\lambda)) + t M_{m t^{\gamma}}(\lambda). \label{chern1}
\end{align}
We now choose $\lambda(t):=c_1(m t^{\gamma})^{\sigma}=c_1m^{\sigma}t^{\gamma \sigma}$. Replacing $\lambda$ by $\lambda(t)$ in \eqref{chern1} we get 
\begin{align}
\frac{\log \left ( \mathbb{P} \left ( \underline{Y}(t) < m t^{\gamma} \right ) \right )}{t^{\gamma(1+\sigma)}} \leq - \left ( c_1^{\alpha} m^{\alpha \sigma} \lambda(t)^{-\alpha} \phi(\lambda(t)) - c_1 m^{1+\sigma} \right ) + t^{1-\gamma(1+\sigma)} M_{m t^{\gamma}}(c_1(m t^{\gamma})^{\sigma}), \label{chern2}
\end{align}
where we have used that, because of the definition of $\gamma$, $t^{\gamma \sigma} \times t^{\gamma} = t \times (t^{\gamma \sigma})^{\alpha}$. It is not difficult to see from the definition of $\gamma$ that $\gamma(1+\sigma)>1$. Using that, \eqref{regbranchlow} from Assumption \ref{regbranch}, and Assumption \ref{stablelike}, we can let $t$ go to infinity into \eqref{chern2} and get 
\begin{align*}
\limsup_{t \rightarrow \infty} \frac{\log \left ( \mathbb{P} \left ( \underline{Y}(t) < m t^{\gamma} \right ) \right )}{t^{\gamma(1+\sigma)}} \leq - \left ( Cc_1^{\alpha} m^{\alpha \sigma} - c_1 m^{1+\sigma} \right ). 
\end{align*}
Since $m \in (0,m_0)$, we have $Cc_1^{\alpha} m^{\alpha \sigma} - c_1 m^{1+\sigma}>0$ so the proof is complete. 
\end{proof}

The proof of \eqref{minoliminf} can now be concluded by a standard Borel-Cantelli argument. 

\begin{proof}[Proof of \eqref{minoliminf}]
Let us fix $m \in (0,2^{-\gamma}m_0)$ where $m_0$ is as in Lemma \ref{boundprobalow} and, for all $n\geq 1$, define the events 
\begin{align*}
A_n:=\left \{\inf_{t \in [n,n+1]} \frac{\underline{Y}(t)}{t^{\gamma}} < m \right \}. 
\end{align*}
Note that $A_n \subset \{\underline{Y}(n) < m (n+1)^{\gamma} \} \subset \{\underline{Y}(n) < 2^{\gamma}m n^{\gamma} \}$ so, since $2^{\gamma}m \in (0,m_0)$, we can apply Lemma \ref{boundprobalow} and get the existence of $\epsilon>0$ such that, for all $n$ large enough, $\mathbb{P}(A_n)\leq e^{-\epsilon n^{\gamma(1+\sigma)}}$. We thus get $\sum_{n \geq 1}\mathbb{P}(A_n)<\infty$ and the result follows by the Borel-Cantelli lemma. 
\end{proof}

\subsection{Upper bound} \label{lowestpartupperbound}

In this section we assume that all assumptions from Theorem \ref{generalcriteria} hold true and prove 
\begin{align}
\mathbb{P}-a.s. \ \limsup_{t \rightarrow \infty} \frac{\underline{Y}(t)}{t^{\gamma}} <\infty. \label{majolimsup}
\end{align}
Fix $m>0$ and $q>0$ that will be determined later and $k_0\geq 1$ such that $m 2^{k_0 \gamma-1}>a_0(\Lambda)$ (see Section \ref{moredef}). We choose a $k \geq k_0$ and set $a=a(k):=m 2^{k \gamma-1}$. We consider $Y^a$ (as described in Section \ref{moredef} and rigorously introduced in Definition \ref{TroncBranchingSubor}) with this choice of $a$. We denote by $\mathcal{U}_k^a$ the set of indices of particles that are born on the time interval $[2^{k-1},2^k)$, ie $\mathcal{U}_k^a:=\{ u \in \mathcal{U}; \mathfrak{b}^a_{u} \in [2^{k-1},2^k) \}$ (see Appendix \ref{construction} for the definition of the set of indices $\mathcal{U}$ and the birth time $\mathfrak{b}^a_{u}$). 
We denote by $S_{br}^a$ the set of times at which particles of $Y^a$ experience branchings in $\mathcal{Q}^{\star}$ (see \eqref{defqstar} or after \eqref{condbranchinglp0}). For $t \in S_{br}^a$ we denote by $u(t) \in \mathcal{U}$ the index of the particle branching at time $t$ and by $\bm{x}(t)=(x_i(t))_{i\geq 1}$ the corresponding element of $\mathcal{Q}^{\star}$, that is, $x_1(t)=Y^a_{u(t)}(t)-Y^a_{u(t)}(t-)$ and for any $i\geq 1$, $x_{i+1}(t)$ is the distance between $Y^a_{u(t)}(t-)$ and the birth position of the $i$-th child of $u(t)$ at time $t$. We define the random sets 
\begin{align*}
S_{low,k} & := \{ t \in [2^k,2^{k+1}) \cap S^{a}_{br}; u(t) \in \mathcal{U}_k^a, Y^a_{u(t)}(t-) = \min_{v \in \mathcal{U}_k^a} Y^a_{v}(t-) \}, \\
\mathcal{D}_k & :=\{ (t,i,\mathfrak{z}); t \in S_{low,k}, i \geq 2, \mathfrak{z}=x_i(t) <\infty \}. 
\end{align*}
The set $\mathcal{D}_k$ contains information on branchings during the time interval $[2^k,2^{k+1})$, where the branching particle is the leftmost one among those that are born on the time interval $[2^{k-1},2^k)$. Lemma \ref{uniqueleftmost} shows that the minimum in the definition of $\mathcal{D}_k$ is almost surely reached by exactly one particle. The following lemma relies on Assumption \ref{nojumpatbranchings} and tells us that the branchings recorded in $S_{low,k}$ and $\mathcal{D}_k$ occur similarly to branchings of a generic particle. Recall the measure $\Lambda^a(\mathrm{d}\bm{x})$ from Section \ref{moredef}. 
%Using Lemma \ref{uniqueleftmost}, the branching property and that all trajectories of particles are c\`ad-l\`ag, it is not difficult to see that, conditionally on $\{\sharp \mathcal{U}_k^a \geq 1\}$, $\min S_{low,k} \sim 2^k + Z \mathds{1}_{Z < 2^k} + \infty \mathds{1}_{Z \geq 2^k}$ where $Z$ is exponentially distributed with parameter $\Lambda(\{\bm{x}\in\mathcal{Q};\; x_2< a\})<\infty$ (recall from \eqref{condbranchinglplight} that $\Lambda(\{\bm{x}\in\mathcal{Q};\; x_2< a\})<\infty$). By the strong branching property (see Remark \ref{ReguBS}), we deduce that 
\begin{lemma}[branchings of the leftmost particle] \label{branchingleftmost}
Conditionally on $\{\sharp \mathcal{U}_k^a \geq 1\}$, the point process $\{ (t,\bm{x}(t)); \ t \in S_{low,k} \}$ on $[2^k,2^{k+1}) \times \mathcal{Q}^{\star}$ is a Poisson point process with intensity measure $\mathrm{d}t \otimes \Lambda^a(\cdot \cap \mathcal{Q}^{\star})$. 
\end{lemma}
The proof of Lemma \ref{branchingleftmost} is straightforward and shifted to the end of Appendix \ref{construction}. Since $\Lambda^a(\mathcal{Q}^{\star})=\Lambda(\{\bm{x}\in\mathcal{Q};\; x_2< a\})<\infty$ by \eqref{condbranchinglplight} and, by \eqref{condbranchinglplight} again, the average number of new particles arriving in $Y^a$ at each jumping time of the above point process is finite. Therefore, $\mathbb{E}[\sharp \mathcal{D}_k]<\infty$ so $\mathcal{D}_k$ is almost surely finite (in particular, $S_{low,k}$ is a discrete subset of $[2^k,2^{k+1})$). We then define 
\begin{align}
S_k & := \sum_{(t,i,\mathfrak{z}) \in \mathcal{D}_k} e^{-\mathfrak{z} q 2^{k \gamma \sigma}}, \label{defsk} \\
B_k &:= \{ S_k \geq 2^{k-1}M_{m 2^{k \gamma-1}}(q 2^{k \gamma\sigma}) \}. \label{defbk}
\end{align}
For any $(t,i,\mathfrak{z}) \in \mathcal{D}_k$, introduce $\bm{\xi}_{t,i,\mathfrak{z}}:=\bm{\xi}^a_{u_{i}(t)}$, where $u_{i}(t)$ is the $(i-1)^{th}$ child born from the particle $u(t)$ at time $t$ (the child born with displacement $x_i(t)=\mathfrak{z}$) and where we recall that, for a particle $v$, $\bm{\xi}^a_v$ is the canonical trajectory (associated with $Y^a$) of the particle $v$, see below \eqref{condbranchinglp0} for a description of canonical trajectories and Definition \ref{DefCanoTraj} for a rigorous definition. In particular, we see from the construction in Appendix \ref{construction} that, conditionally on $\mathcal{D}_k$, $(\bm{\xi}_{t,i,z})_{(t,i,\mathfrak{z})\in \mathcal{D}_k}$ is a collection of\textrm{ i.i.d }copies of $(\bm{\xi}(t); t\geq 0)$, a subordinator with Laplace exponent $\phi$ and starting position $0$. For any $k\geq k_0$ we let 
\begin{align}
C_k := \{ \min \{ \mathfrak{z}+\bm{\xi}_{t,i,\mathfrak{z}}(2^{k+2}); (t,i,\mathfrak{z})\in \mathcal{D}_k \} \leq m2^{k \gamma} \}, \label{defck}
\end{align}
with the convention $\min \emptyset =\infty$.

\begin{lemma} \label{boundprobahigh1}
If $2^{\sigma}q/m^{\sigma}\leq c_0$ (where $c_0$ is as in Assumption \ref{controlvar}) then we have $\sum_{k \geq k_0} \mathbb{P}(B_k^c)<\infty$. 
\end{lemma}

\begin{proof}
Fix $k \geq k_0$ and set $a:=m 2^{k \gamma-1}$ as before. Define $F_k:\mathcal{Q}^{\star} \rightarrow \mathbb{R}_+$ by $F_k(\bm{x}):=\sum_{i \geq 2} e^{-x_i q 2^{k \gamma \sigma}}\mathds{1}_{\{x_i<a\}}$ and note that $S_k=\sum_{t \in S_{low,k}} F_k(\bm{x}(t))$. We know from Lemma \ref{branchingleftmost} that, conditionally on $\{\sharp \mathcal{U}_k^a \geq 1\}$, $\{ (t,\bm{x}(t)); \ t \in S_{low,k} \}$ is a Poisson point process on $[2^k,2^{k+1}) \times \mathcal{Q}^{\star}$ with intensity measure $\mathrm{d}t \otimes \Lambda^a(\cdot \cap \mathcal{Q}^{\star})$. 
In particular, if we let $(X_k^i)_{i\geq 1}$ be\textrm{ i.i.d }random variables distributed as $F_k(\bm{z})$ (where $\bm{z}$ has distribution $\Lambda^a(\cdot \cap \mathcal{Q}^{\star})/\Lambda^a(\mathcal{Q}^{\star})$ on $\mathcal{Q}^{\star}$), and $N_k$ be a Poisson random variable with parameter $2^k \Lambda^a(\mathcal{Q}^{\star})$, with $\{\sharp \mathcal{U}_k^a \geq 1\}$, $N_k$ and $(X_k^i)_{i\geq 1}$ being independent, then $S_k$ is distributed as $\mathds{1}_{\sharp \mathcal{U}_k^a \geq 1} \tilde S_k$ with $\tilde S_k:=\sum_{j=1}^{N_k}X_k^j$. Using classical identities for compound Poisson distributions and the definitions of $M_{\cdot}(\cdot)$ and $M_{2,\cdot}(\cdot)$ in \eqref{laplaceexpospecialcase2} and \eqref{defm2} we get 
\begin{align*}
\mathbb{E}[\tilde S_k]&=\mathbb{E}[N_k] \times \mathbb{E}[X_k^1] = 2^k M_{m 2^{k \gamma-1}}(q 2^{k \gamma\sigma}), \\ 
Var(\tilde S_k)&=\mathbb{E}[N_k] \times \mathbb{E}[(X_k^1)^2] = 2^k M_{2, m 2^{k \gamma-1}}(q 2^{k \gamma\sigma}). 
\end{align*}
Then, by definition of $B_k$ in \eqref{defbk} and Chebyshev inequality we get 
\begin{align*}
\mathbb{P}(B_k^c) & \leq \mathbb{P}(\tilde S_k \leq \mathbb{E}[\tilde S_k]/2) + \mathbb{P}(\sharp \mathcal{U}_k^a =0) \leq \mathbb{P}(|S_k - \mathbb{E}[S_k]| \geq \mathbb{E}[S_k]/2) + \mathbb{P}(\sharp \mathcal{U}_k^a =0) \\
& \leq \frac{4 Var(S_k)}{\mathbb{E}[S_k]^2} + \mathbb{P}(\sharp \mathcal{U}_k^a =0) = 2^{-k} \frac{4 M_{2, m 2^{k \gamma-1}}(q 2^{k \gamma\sigma})}{(M_{m 2^{k \gamma-1}}(q 2^{k \gamma\sigma}))^2} + \mathbb{P}(\sharp \mathcal{U}_k^a =0). 
\end{align*}
Since the initial particle $\varnothing$ of $Y^a$ gives birth to new particles at rate $\Lambda^a(\mathcal{Q}^{\star})=\Lambda(\{ \bm{x}\in\mathcal{Q}; x_2<a \}) \geq \Lambda(\{ \bm{x}\in\mathcal{Q}; x_2<m 2^{k_0 \gamma-1} \})>0$ we get $\mathbb{P}(\sharp \mathcal{U}_k^a =0)\leq e^{-2^{k-1} \Lambda(\{ \bm{x}\in\mathcal{Q}; x_2<m 2^{k_0 \gamma-1} \})}$. Since $2^{\sigma}q/m^{\sigma}\leq c_0$, \eqref{controlvareq} from Assumption \ref{controlvar} holds true with $c=2^{\sigma}q/m^{\sigma}$. We get that the above is, for large $k$, bounded by $2^{-k} L$ for some constant $L>0$. This yields $\sum_{k \geq k_0} \mathbb{P}(B_k^c)<\infty$. 
\end{proof}

\begin{lemma} \label{boundprobahigh2}
There is a choice of the constants $m$ and $q$ in \eqref{defsk}-\eqref{defck} such that we have $\sum_{k \geq k_0} \mathbb{P}(C_k^c)<\infty$. 
\end{lemma}

\begin{proof}
Recall that $\bm{\xi}$ denotes a subordinator with Laplace exponent $\phi$ and starting position $0$. We let ${P}(\cdot)$ be the probability measure associated with $\bm{\xi}$ and ${E}[\cdot]$ the associated expectation. For $k\geq k_0$ we have 
\begin{align}
\mathbb{P}(C_k^c) & \leq \mathbb{P} \left ( \forall (t,i,\mathfrak{z})\in \mathcal{D}_k, \mathfrak{z}+\bm{\xi}_{t,i,\mathfrak{z}}(2^{k+2}) > m2^{k \gamma} \right ) \nonumber \\
& = \mathbb{E} \left [ \prod_{(t,i,\mathfrak{z})\in \mathcal{D}_k} \left ( 1-{P} \left (  \mathfrak{z}+\bm{\xi}(2^{k+2}) \leq m2^{k \gamma} \right ) \right ) \right ] \nonumber \\
& \leq \mathbb{E} \left [ \exp \left ( -\sum_{(t,i,\mathfrak{z})\in \mathcal{D}_k} {P} \left (  \mathfrak{z}+\bm{\xi}(2^{k+2}) \leq m2^{k \gamma} \right ) \right ) \right ], \label{probaminislarge}
\end{align}
where, for the last inequality, we have used that $1-z\leq e^{-z}$. Note that we use the conventions $\prod_{(t,i,\mathfrak{z})\in \emptyset}=1$ and $\sum_{(t,i,\mathfrak{z})\in \emptyset}=0$. 

We now study $\sum_{(t,i,\mathfrak{z})\in \mathcal{D}_k} {P}(\mathfrak{z}+\bm{\xi}(2^{k+2}) \leq m2^{k \gamma})$ using a change of measure argument that is used in the proof of Theorem III.11 of \cite{Bertoin} and that comes from large deviations theory. For $\omega\in(0,\infty)$, let ${P}^{(\omega)}(\cdot)$ be defined by ${P}^{(\omega)}(A):={E} \left [ e^{-\omega \bm{\xi}(t) + \phi(\omega)t} \mathds{1}_A \right ]$ for $A \in \mathcal{F}_t$ and $t \geq 0$, where $\mathcal{F}_t$ is the sigma field generated by $(\bm{\xi}(s))_{s \in [0,t]}$. We see from Section III.4 of \cite{Bertoin} that this definition is consistent and that, under ${P}^{(\omega)}(\cdot)$, $\bm{\xi}$ is a subordinator with Laplace exponent $\phi^{(\omega)}(\cdot):=\phi(\omega+\cdot)-\phi(\omega)$ so, in particular, 
\begin{align}
{E}^{(\omega)}[\bm{\xi}(t)]=t\phi'(\omega) \ \ \ \text{and} \ \ \ Var^{(\omega)}[\bm{\xi}(t)]=-t\phi''(\omega). \label{exprespvar}
\end{align}
Note that $\phi'(\cdot)$ is continuous and decreasing from $(0,\infty)$ to $(0,\phi'(0))$, where $\phi'(0)\in (0,\infty]$. In particular, $\phi'$ admits an inverse function $(\phi')^{-1}:(0,\phi'(0))\rightarrow (0,\infty)$. There is $k_0(m)\geq k_0$ such that for all $k\geq k_0(m)$ we have $m2^{k \gamma}/2^{k+2}<\phi'(0)$. We assume that $k\geq k_0(m)$. For any $\mathfrak{z} \in (0,m2^{k\gamma -1})$ we have 
\begin{align}
{P} \left ( \mathfrak{z}+\bm{\xi}(2^{k+2}) \leq m2^{k \gamma} \right ) & = {E}^{(\omega)} \left [ e^{\omega\bm{\xi}(2^{k+2}) - \phi(\omega) 2^{k+2}} \mathds{1}_{\{\mathfrak{z}+\bm{\xi}(2^{k+2}) \leq m2^{k \gamma}\}} \right ] \nonumber \\
 \geq &e^{\omega(m2^{k \gamma}-\mathfrak{z})/2 - \phi(\omega) 2^{k+2}} \nonumber \\
& \times {P}^{(\omega)} \left ((m2^{k \gamma}-\mathfrak{z})/2 \leq\bm{\xi}(2^{k+2}) \leq m2^{k \gamma}-\mathfrak{z} \right ). \label{cramermino}
\end{align}
Since $m2^{k \gamma}/2^{k+2}<\phi'(0)$ we can choose $\omega_{\mathfrak{z}}:=(\phi')^{-1}(3\times 2^{-(k+4)}(m2^{k \gamma}-\mathfrak{z}))$. Combining with \eqref{exprespvar} we get ${E}^{(\omega_{\mathfrak{z}})}[\bm{\xi}(2^{k+2})]=3(m2^{k \gamma}-\mathfrak{z})/4$. Therefore, 
\begin{align}
{P}^{(\omega_{\mathfrak{z}})} \left ((m2^{k \gamma}-\mathfrak{z})/2 \leq \bm{\xi}(2^{k+2}) \leq m2^{k \gamma}-\mathfrak{z} \right ) = & 1-{P}^{(\omega_{\mathfrak{z}})} \left ( \left | \bm{\xi}(2^{k+2}) - {E}^{(\omega_{\mathfrak{z}})}[\bm{\xi}(2^{k+2})] \right | > (m2^{k \gamma}-\mathfrak{{z}})/4 \right ) \nonumber \\
\geq & 1-2^4 Var^{(\omega_{\mathfrak{z}})} (\bm{\xi}(2^{k+2}))/(m2^{k \gamma}-\mathfrak{z})^2 \nonumber \\
%= & 1+2^{k+2}\phi''((\phi')^{-1}(2^{-(k+2)}(1-\epsilon)(m2^{k \gamma}-z)))/\epsilon^2(m2^{k \gamma}-z)^2 \\
= & 1+2^{k+6}\phi''(\omega_{\mathfrak{z}})/(m2^{k \gamma}-\mathfrak{z})^2, \label{phi2}
\end{align}
where we have used Chebyshev inequality and \eqref{exprespvar}. 
From Assumption \ref{stablelike} and the monotone density theorem we have that $\phi'(\lambda)\sim \alpha C \lambda^{\alpha-1}$ and $\phi''(\lambda)\sim -\alpha (1-\alpha) C \lambda^{\alpha-2}$ as $\lambda$ goes to infinity so, in particular, $(\phi')^{-1}(r)\sim (\alpha C/r)^{1/(1-\alpha)}$ as $r$ goes to zero. We get that $-\phi''((\phi')^{-1}(r)) \sim C'r^{(2-\alpha)/(1-\alpha)}$ (where $C':=(1-\alpha) (\alpha C)^{-1/(1-\alpha)}$) so there is $r_0>0$ such that $-\phi''((\phi')^{-1}(r)) \leq 2C'r^{(2-\alpha)/(1-\alpha)}$ for all $r\leq r_0$. We choose $k_1(m)\geq k_0(m)$ such that $m2^{k \gamma}/2^{k+2}\leq r_0$ for all $k\geq k_1(m)$. We get, whenever $k\geq k_1(m)$, for any $\mathfrak{z} \in (0,m2^{k\gamma -1})$, that $-\phi''(\omega_{\mathfrak{z}})\leq 2C'(3\times 2^{-(k+4)}(m2^{k \gamma}-\mathfrak{z}))^{(2-\alpha)/(1-\alpha)}$. 
In this case, the last term in \eqref{phi2} is larger than 
\begin{align*}
& 1-3^{(2-\alpha)/(1-\alpha)} 2^3 C' 2^{-\frac{k+4}{1-\alpha}} (m2^{k \gamma}-\mathfrak{z})^{\frac{\alpha}{1-\alpha}} \geq 1-3^{(2-\alpha)/(1-\alpha)} 2^{3-\frac{4}{1-\alpha}} C' m^{\frac{\alpha}{1-\alpha}} 2^{\frac{ k(\gamma \alpha -1)}{1-\alpha}}. 
\end{align*}
By the definition of $\gamma$ we have $\gamma \alpha<1$ so there is $k_2(m)\geq k_1(m)$ such that, for all $k\geq k_2(m)$, the above is larger that $1/2$. We thus get that for all $k \geq k_2(m)$ and $\mathfrak{z}\in (0,m2^{k\gamma -1})$, we have 
\begin{align*}
{P}^{(\omega_{\mathfrak{z}})} \left ((m2^{k \gamma}-\mathfrak{z})/2 \leq \bm{\xi}(2^{k+2}) \leq m2^{k \gamma}-\mathfrak{z} \right ) \geq 1/2. 
\end{align*}
Plugging this into \eqref{cramermino} and summing over $\mathcal{D}_k$ we get that for any $k \geq k_2(m)$, 
\begin{align}
\sum_{(t,i,\mathfrak{z})\in \mathcal{D}_k} {P} \left ( \mathfrak{z}+\bm{\xi}(2^{k+2}) \leq m2^{k \gamma} \right ) \geq & \frac1{2} \sum_{(t,i,\mathfrak{z})\in \mathcal{D}_k} e^{\omega_{\mathfrak{z}}(m2^{k \gamma}-\mathfrak{z})/2 - \phi(\omega_{\mathfrak{z}}) 2^{k+2}}. \label{cramerminosum1} 
\end{align}
From the behaviors of $\phi$ and $(\phi')^{-1}$, we have $(\phi')^{-1}(r)\sim C''r^{-1/(1-\alpha)}$ and $\phi((\phi')^{-1}(r)) \sim C'''r^{-\alpha/(1-\alpha)}$ as $r$ goes to zero, where $C'':=(\alpha C)^{1/(1-\alpha)}$ and $C''':=C^{1/(1-\alpha)} \alpha^{\alpha/(1-\alpha)}$. In particular, there is $r_1>0$ such that $(\phi')^{-1}(r)\geq C''r^{-1/(1-\alpha)}/2$ and $\phi((\phi')^{-1}(r)) \leq 2C'''r^{-\alpha/(1-\alpha)}$ for all $r\leq r_1$. We choose $k_3(m)\geq k_2(m)$ such that $m2^{k \gamma}/2^{k+2}\leq r_1$ for all $k\geq k_3(m)$. We get, whenever $k\geq k_3(m)$, for any $\mathfrak{z} \in (0,m2^{k\gamma -1})$, that $\omega_{\mathfrak{z}}\geq C''(2^{-(k+2)}(m2^{k \gamma}-\mathfrak{z}))^{-1/(1-\alpha)}/2$ and $\phi(\omega_{\mathfrak{z}})\leq 2C''' (2^{-(k+3)}(m2^{k \gamma}-\mathfrak{z}))^{-\alpha/(1-\alpha)}$. In this, case the right-hand side of \eqref{cramerminosum1} is larger than 
\begin{align}
\frac1{2} \sum_{(t,i,\mathfrak{z})\in \mathcal{D}_k} e^{-\tilde C(m2^{k \gamma}-\mathfrak{z})^{-\frac{\alpha}{1-\alpha}}2^{\frac{k}{1-\alpha}}} = \frac1{2} \sum_{(t,i,\mathfrak{z})\in \mathcal{D}_k} e^{-\tilde Cm^{-\frac{\alpha}{1-\alpha}}(1-\frac{\mathfrak{z}}{m 2^{k \gamma}})^{-\frac{\alpha}{1-\alpha}}2^{\frac{k(1-\alpha \gamma)}{1-\alpha}}}, \label{cramerminosum}
\end{align}
where $\tilde C:=C'''2^{3+3\alpha/(1-\alpha)}-C''2^{-2+2/(1-\alpha)}>0$. In all the above terms, we have $\mathfrak{z}/m 2^{k \gamma}<1/2$ so $(1-\mathfrak{z}/m 2^{k \gamma})^{-\frac{\alpha}{1-\alpha}} \leq 1+K\mathfrak{z}/m 2^{k \gamma}$ where $K:=\sup_{\mathfrak{t} \in (0,1/2]} \mathfrak{t}^{-1}((1-\mathfrak{t})^{-\frac{\alpha}{1-\alpha}}-1)<\infty$. Plugging this into \eqref{cramerminosum}, we get 
\begin{align}
& \sum_{(t,i,\mathfrak{z})\in \mathcal{D}_k} {P} \left ( \mathfrak{z}+\bm{\xi}(2^{k+2}) \leq m2^{k \gamma} \right ) \geq \frac1{2} e^{-\tilde C m^{-\frac{\alpha}{1-\alpha}} 2^{\frac{ k(1-\alpha \gamma)}{1-\alpha}}} \sum_{(t,i,\mathfrak{z})\in \mathcal{D}_k} e^{-\mathfrak{z} \tilde C m^{-\frac{1}{1-\alpha}} K 2^{\frac{k(1-\gamma)}{1-\alpha}}}. \label{identifysk}
%\geq & \frac1{2} e^{-K(\alpha,\epsilon)c^{-\frac{\alpha}{1-\alpha}}2^{\frac1{1-\alpha}}2^{\frac{k(1-\alpha \gamma)}{1-\alpha}}} e^{-\theta_k c 2^{k \gamma}/2} e^{\delta (c/2)^{1/\kappa} 2^{k \gamma/\kappa}}. 
\end{align}
Note from the definition of $\gamma$ that $(1-\gamma)/(1-\alpha)=\gamma \sigma$. Therefore, if we set 
\begin{align}
q:=\tilde C m^{-\frac{1}{1-\alpha}} K, \label{defq(m)}
\end{align}
in \eqref{defsk} and \eqref{defbk}, then the right hand side of \eqref{identifysk} equals $e^{-\tilde C m^{-\frac{\alpha}{1-\alpha}} 2^{\frac{ k(1-\alpha \gamma)}{1-\alpha}}} S_k/2$. From this and \eqref{defbk} we get that, for $k\geq k_3(m)$, on the event $B_k$ we have almost surely 
\begin{align*}
\sum_{(t,i,\mathfrak{z})\in \mathcal{D}_k} {P} \left ( \mathfrak{z}+\bm{\xi}(2^{k+2}) \leq m2^{k \gamma} \right ) \geq 2^{k-2} e^{-\tilde C m^{-\frac{\alpha}{1-\alpha}} 2^{\frac{ k(1-\alpha \gamma)}{1-\alpha}}} M_{m 2^{k \gamma-1}}(q 2^{k \gamma\sigma}). 
\end{align*}
By Lemma \ref{equivassumption} of Appendix \ref{equivassump}, Assumption \ref{regbranchoriginal} is equivalent to Assumption \ref{regbranch} so we can assume that the latter holds true. Let us assume that $m \geq (\tilde C 2^{\sigma} K/(c_0\wedge c_2))^{\frac{1-\alpha}{1+\sigma(1-\alpha)}}$, where $c_2$ is as in Assumption \ref{regbranch} and $c_0$ as in Assumption \ref{controlvar}. For such an $m$ and for $q$ chosen according to \eqref{defq(m)}, we have $2^{\sigma}q/m^{\sigma}\leq c_0\wedge c_2$. By \eqref{regbranchhigh} from Assumption \ref{regbranch}, there exist $\delta>0$ and $\Bar{a}>0$ such that, whenever $a\geq\Bar{a}$ we have $M_{a}(\lambda) \geq e^{\delta a^{1+\sigma}}$ for all $\lambda \leq c_2 a^{\sigma}$. 
%$\theta_k:=K(\alpha,\epsilon)c^{-\frac{1}{1-\alpha}}2^{\frac{2}{1-\alpha}} M 2^{\frac{k(1-\gamma)}{1-\alpha}}$. 
%Let us fix $q>2^{\frac{2}{1-\alpha}}/(2^{\frac{2}{1-\alpha}-1-\beta} M)$. 
%We assume that the constant $c$ in \eqref{defck} is chosen large enough so that $c^{\frac{1}{1-\alpha}+\frac{1-\beta}{\beta}}\geq (1+q) K(\alpha,\epsilon)2^{1-\beta+\frac{2}{1-\alpha}+\frac{1-\beta}{\beta}}M$. Using that, by definition of $\gamma$, we have $2^{\frac{k(1-\gamma)}{1-\alpha}}=2^{\frac{k\gamma(1-\beta)}{\beta}}$ we get that, for such a choice of $c$, we have $\theta_k \leq (c2^{k \gamma}/2)^{(1-\beta)/\beta}/(1+q)2^{1-\beta}$. We can thus use Lemma \ref{upperboundexponent} and get 
In conclusion, for our choices of $m$ and $q$, if $k\geq k_4(m) := k_3(m) \vee \lfloor 1+\log(2\Bar{a}/m)/\gamma\log(2) \rfloor$, we have almost surely on $B_k$ that 
\begin{align}
\sum_{(t,i,\mathfrak{z})\in \mathcal{D}_k} {P} \left ( \mathfrak{z}+\bm{\xi}(2^{k+2}) \leq m2^{k \gamma} \right ) & \geq 2^{k-2} e^{L 2^{\frac{k(1-\alpha \gamma)}{1-\alpha}}}, \label{minosumprobysmall}
\end{align}
where we have used that, by definition of $\gamma$, $\gamma(1+\sigma)=(1-\alpha \gamma)/(1-\alpha)$, and where we have set $L:=-\tilde C m^{-\frac{\alpha}{1-\alpha}} + \delta (m/2)^{1+\sigma}$. 
%\begin{align*}
%L:=&-\tilde C m^{-\frac{\alpha}{1-\alpha}} + \delta (m/2)^{1+\sigma}. 
%\end{align*}
By increasing $m$ if necessary, we can assume that we have $m>(\tilde C 2^{1+\sigma}/\delta)^{\frac{1-\alpha}{1+\sigma(1-\alpha)}}$ so, in particular, $L>0$. Recall that, by our choice of $m$ and $q$ we have $2^{\sigma}q/m^{\sigma}\leq c_0\wedge c_2$ so Lemma \ref{boundprobahigh1} is applicable. Therefore, plugging \eqref{minosumprobysmall} into \eqref{probaminislarge} we get that, for our choice of $m$ and $q$, 
\begin{align*}
\sum_{k \geq k_0} \mathbb{P}(C_k^c) \leq k_4(m) + \sum_{k\geq k_4(m)} \left ( e^{-2^{k-2}} + \mathbb{P}(B_k^c) \right ) <\infty, 
\end{align*}
which concludes the proof. 
\end{proof}

We can now conclude the proof of \eqref{majolimsup}. 

\begin{proof}[Proof of \eqref{majolimsup}]
Since $C_k \subset \{\mathcal{D}_k \neq \emptyset\}$, we have almost surely that, on $C_k$, $\{ \min_{v \in \mathcal{U}_k^{m 2^{k \gamma-1}}} Y^{m 2^{k \gamma-1}}_{v}(t-) +\mathfrak{z}+\bm{\xi}_{t,i,\mathfrak{z}}(2^{k+2}-t); (t,i,\mathfrak{z})\in \mathcal{D}_k \}$ is a non-empty subset of the set of positions at time $2^{k+2}$ of particles in the system $Y^{m 2^{k \gamma-1}}$ that are born in the time interval $[2^k,2^{k+1})$. In particular, we have almost surely that, on $C_k$, 
\begin{align}
\min_{v \in \mathcal{U}_{k+1}^{m 2^{k \gamma-1}}} Y^{m 2^{k \gamma-1}}_{v}(2^{k+2}) & \leq \min_{(t,i,\mathfrak{z})\in \mathcal{D}_k} \left \{ \min_{v \in \mathcal{U}_k^{m 2^{k \gamma-1}}} Y^{m 2^{k \gamma-1}}_{v}(t-)+\mathfrak{z}+\bm{\xi}_{t,i,\mathfrak{z}}(2^{k+2}-t) \right \} \nonumber \\
& \leq \min_{(t,i,\mathfrak{z})\in \mathcal{D}_k} \left \{ \min_{v \in \mathcal{U}_k^{m 2^{k \gamma-1}}} Y^{m 2^{k \gamma-1}}_{v}(2^{k+1})+\mathfrak{z}+\bm{\xi}_{t,i,\mathfrak{z}}(2^{k+2}) \right \} \nonumber \\
& = \min_{v \in \mathcal{U}_k^{m 2^{k \gamma-1}}} Y^{m 2^{k \gamma-1}}_{v}(2^{k+1})+\min_{(t,i,\mathfrak{z})\in \mathcal{D}_k} \{ \mathfrak{z}+\bm{\xi}_{t,i,\mathfrak{z}}(2^{k+2}) \} \nonumber \\ 
& \leq \min_{v \in \mathcal{U}_k^{m 2^{(k-1) \gamma-1}}} Y^{m 2^{(k-1) \gamma-1}}_{v}(2^{k+1})+m2^{k \gamma}. \label{inegbk}
\end{align}
We chose $m$ and $q$ as in Lemma \ref{boundprobahigh2}. By that Lemma and the Borel-Cantelli lemma, there exists almost surely a finite random index $K_0\geq k_0$ such that $C_k$ is realized for all $k\geq K_0$. Iterating \eqref{inegbk} we get that, almost surely, for all $k\geq K_0$ and $s \in [2^{k},2^{k+1}]$, we have 
\begin{align*}
\underline{Y}(s) \leq \underline{Y}(2^{k+1}) \leq \min_{v \in \mathcal{U}_k^{m 2^{(k-1) \gamma-1}}} Y^{m 2^{(k-1) \gamma-1}}_{v}(2^{k+1}) \leq \min_{v \in \mathcal{U}_{K_0}^{m 2^{(K_0-1) \gamma-1}}} Y_v^{m 2^{(K_0-1) \gamma-1}}(2^{K_0+1}) + \sum_{j=K_0}^{k-1} m2^{j\gamma}. 
\end{align*}
Note that $\sum_{j=K_0}^{k-1} m2^{j\gamma} = m(2^{k \gamma} - 2^{K_0 \gamma})/(2^{\gamma}-1)$ so, almost surely, for all $k\geq K_0$ we have 
\begin{align*}
\sup_{s \in [2^{k},2^{k+1}]} \frac{\underline{Y}(s)}{s^{\gamma}} \leq \frac{m}{2^{\gamma}-1} + 2^{-k \gamma} \left ( \min_{v \in \mathcal{U}_{K_0}^{m 2^{(K_0-1) \gamma-1}}} Y_v^{m 2^{(K_0-1) \gamma-1}}(2^{K_0+1}) - \frac{m 2^{K_0 \gamma}}{2^{\gamma}-1} \right ). 
\end{align*}
We deduce that, almost surely, 
\begin{align*}
\limsup_{k \rightarrow \infty} \sup_{s \in [2^{k},2^{k+1}]} \frac{\underline{Y}(s)}{s^{\gamma}} \leq \frac{m}{2^{\gamma}-1}, 
\end{align*}
and \eqref{majolimsup} follows. 
\end{proof}

\begin{appendix}
\section{Construction of branching subordinators} \label{construction}

Our construction borrows most of its arguments from \cite{Shi_Watson} and \cite{Bertoin_Mallein2019} but we slightly adapt their construction so that it fits our needs. We shall not give all details and we may refer to original arguments when needed. Before getting started, let us recall a few notations and definitions. The set
\begin{align}
    \mathcal{Q}^{\star}=\{\bm{x}\in\mathcal{Q};\; x_2<\infty\} \label{defqstar}
\end{align}
stands for the set of sequences of $\mathcal{Q}$ with at least two finite terms and we assume that $\Lambda(\mathcal{Q}^{\star})\in(0,\infty]$, see \eqref{condbranchinglp0}. Recall (see equation \eqref{condbranchinglplight}) that the integral $\int_{\mathcal{Q}} \big|\sharp \{n \geq 1;\; x_n < a \}-1\big| \Lambda(\mathrm{d}\bm{x})$ is finite for any $a>0$. This hypothesis aims to replace the classical exponential assumption \eqref{condbranchinglp3} and is enough to ensure that for any $a>0$ and $\lambda\geq0$
\begin{align*}
    M_a(\lambda)=\int_{\mathcal{Q}} \left ( \sum_{k=2}^{\infty} e^{-\lambda x_k} \mathds{1}_{\{x_k<a\}} \right ) \Lambda(\mathrm{d}\bm{x})<\infty,
\end{align*}
and we recall that 
\begin{align}
\phi(\lambda) = d\lambda + \int_{\mathcal{Q}} \left ( 1-e^{-\lambda x_1} \right ) \Lambda(\mathrm{d}\bm{x})=d\lambda + \int_0^{\infty} \left ( 1-e^{-\lambda\mathfrak{z}} \right ) \Lambda_1(\mathrm{d}\mathfrak{z})<\infty, \label{seconddefphi}
\end{align}
where $\Lambda_1$ denotes the image of $\Lambda$ by the projection $\bm{x}=(x_n)_{n\geq 1}\mapsto x_1$. In particular, $\phi$ is the Laplace exponent of a subordinator with drift $d$ and L\'evy measure $\Lambda_1$. \\
Also recall that $T_a:\bm{x}\in\mathcal{Q}\mapsto \bm{x}^a\in\mathcal{Q}$ is defined by $(\bm{x}^a)_1=x_1$ and for any $i\geq 2$, $(\bm{x}^a)_i=x_i \mathds{1}_{\{x_i<a\}}+\infty \mathds{1}_{\{x_i\geq a\}}$, and that $\Lambda^a(\mathrm{d}\bm{x})=\mathds{1}_{\{\bm{x}\neq (0,\emptyset)\}}(T_a \Lambda)(\mathrm{d}\bm{x})$, where $T_a \Lambda$ denotes the image measure of $\Lambda$ by $T_a$. We see from the definition of $\Lambda^a$ and \eqref{condbranchinglplight} that $\Lambda^a$ is supported on $\{\bm{x}\in\mathcal{Q};\; \sharp \{ i \geq 2; x_i<\infty\} <\infty\}$. Recall that $\Lambda^a(\mathcal{Q}^{\star})=\Lambda(\{ \bm{x}\in\mathcal{Q}; x_2<a \})$ so $\lim_{a\to\infty}\Lambda^a(\mathcal{Q}^{\star})=\Lambda(\mathcal{Q}^{\star})>0$ (see equation \eqref{condbranchinglp0}) so there exists $a_0(\Lambda)\geq 0$ such that $\Lambda^a(\mathcal{Q}^{\star})>0$ for all $a>a_0(\Lambda)$. On the other hand, $\Lambda^a(\mathcal{Q}^{\star})\leq M_a(0)<\infty$. In particular, the measure $\Lambda^{a,\star}:=\Lambda^a(\cdot\cap\mathcal{Q}^{\star})/\Lambda^a(\mathcal{Q}^{\star})$ defined on $\mathcal{Q}^{\star}$ is a probability measure for all $a> a_0(\Lambda)$. From now one we always consider such an $a$. Now, let us introduce
\begin{align}
    \phi^a(\lambda):=d\lambda + \int_{\mathcal{Q}\setminus\mathcal{Q}^{\star}} \left ( 1-e^{-\lambda x_1} \right ) \Lambda^a(\mathrm{d}\bm{x})=d\lambda + \int_0^{\infty} \left ( 1-e^{-\lambda\mathfrak{z}} \right ) \Lambda^a_1(\mathrm{d}\mathfrak{z})<\infty, \label{defphia}
\end{align}
where $\Lambda^a_1$ denotes the image measure of $\Lambda^a(\cdot\cap(\mathcal{Q}\setminus\mathcal{Q}^{\star}))$ by the projection $\bm{x}=(x_n)_{n\geq 1}\mapsto x_1$. The fact that $\phi^a(\lambda)<\infty$ for any $\lambda\geq 0$ comes, again, from the assumption $\int_{\mathcal{Q}}(1\wedge x_1)\Lambda(\mathrm{d}\bm{x})$. In particular, $\phi^a$ is the Laplace exponent of a subordinator with drift $d$ and L\'evy measure $\Lambda^a_1$. \\
We are now ready to construct our branching subordinators. Let $\N^*:=\{1,2,\ldots\}$ and denote by $\mathcal{U}:=\bigcup_{n\geq 0}((\N^*)^3)^n$ the set of finite sequences of $\N^*$-valued $3$-tuples with the convention $((\N^*)^3)^0=\{\varnothing\}$, where $\varnothing$ is the empty sequence. For a given $u\in\mathcal{U}$, there exists $n\geq 0$ such that $u\in((\N^*)^3)^n$ and $n$ stands for the generation of the vertex $u$, denoted by $|u|$. In particular, the generation $0$ only contains $\varnothing$. Any child of $u$ is denoted by $u\bm{p}$ instead of $(u,\bm{p})$ for some $\bm{p}\in(\N^*)^3$. The tree $\mathcal{U}$ will play the role of indexing set for the particles of the branching subordinator $Y^a$ that we are about to construct.
Precisely, for the child $u\bm{p}$ of $u$ with $\bm{p}=(q,k,i)$, the number $q$ will denote the level of this child, that is $q=1+\lfloor d_{u\bm{p}}\rfloor$ where $ d_{u\bm{p}}$ stands for the distance between the birth position of the particle $u\bm{p}$ and the pre-branching position of the particle $u$. The number $k$ will mean that the particle $u\bm{p}$ was born during the $k$-th branching of the particle $u$ such that there exists a child $v$ of particle $u$ with level $q$ and the number $i$ will mean that particle $u\bm{p}$ is the $i$-th child of particle $u$ (that is the child with $i$-th smallest birth position) with level $q$ born during the latter branching. We note that some indices $u\in\mathcal{U}$ will possibly correspond to particles that are never born. An important aspect of the construction of $Y^a$ is that, almost surely, for all $u\in\mathcal{U}$, the sequence of branching times of the particle $u$ will be discrete, thus making possible the indexing of particles by the tree $\mathcal{U}$. That property is not true in general for the branching subordinator $Y$, which will therefore be constructed from $Y^a$ by a limiting procedure. 
We now proceed with the construction.

\vspace{0.1cm}

\noindent\textbf{Step 1}: the truncated branching subordinator $Y^a$.

\vspace{0.1cm}

\noindent Let $a> a_0(\Lambda)$ be a positive integer and introduce two independent families of random variables:
\begin{itemize}
    \item $(\xi_u)_{u\in\mathcal{U}}$, a family of\textrm{ i.i.d }subordinators with drift $d$ and L\'evy measure $\Lambda^a_1$;
    \item $(\mathcal{P}_u)_{u\in\mathcal{U}}$, a family of\textrm{ i.i.d }copies of a Poisson point process on $[0,\infty) \times \mathcal{Q}^{\star}$ with intensity $\mathrm{d}t \otimes \Lambda^a(\cdot\cap\mathcal{Q}^{\star})$. For any borelian set $A\subset[0,\infty)$, introduce $N_u(A):=\sum_{\mathfrak{z}\in\mathcal{P}_u}\un_{\{\mathfrak{z}\in A \times \mathcal{Q}^{\star}\}}$. Note that $(N_u([0,t]))_{t\geq 0}$ is a Poisson process with parameter $\Lambda^a(\mathcal{Q}^{\star})$. 
\end{itemize}
We first deal with the initial particle $\varnothing$. The birth-time $\mathfrak{b}^a_{\varnothing}$ of particle $\varnothing$ equals $0$. Enumerating the points of $\mathcal{P}_{\varnothing}$ by increasing first coordinate we obtain a sequence $(T_{\varnothing,j},\Delta_{\varnothing,j})_{j\geq 1}$ and note that $(\Delta_{\varnothing,j})_{j\geq 1}$ is a family of\textrm{ i.i.d } $\mathcal{Q}^{\star}$-valued random variables with law $\Lambda^{a,\star}$. The $i$-th term of the sequence $\Delta_{\varnothing,j}$ is denoted by $\Delta_{\varnothing,j}(i)$. For any $t\geq 0$ we set, 
\begin{align}
    Y^a_{\varnothing}(t):=\xi_{\varnothing}(t)+\sum_{j=1}^{N_{\varnothing}^0([0,t])}\Delta_{\varnothing,j}(1), \label{defyaempty}
\end{align}
thus giving the generation $0$. Now, let $n\geq 0$, assume that $\{(\mathfrak{b}^a_u,Y^a_u);\; |u|\leq n\}$ has been built, where $\mathfrak{b}^a_u$ denotes the birth-time of particle $u$. We now build generation $n+1$. For $u$ in generation $n$, enumerating the points of $\mathcal{P}_u$, restricted to $[\mathfrak{b}^a_u,\infty) \times \mathcal{Q}^{\star}$, by increasing first coordinate we obtain a sequence $(T_{u,j},\Delta_{u,j})_{j\geq 1}$. Again, $(\Delta_{u,j})_{j\geq 1}$ is a family of\textrm{ i.i.d } $\mathcal{Q}^{\star}$-valued random variables with law $\Lambda^{a,\star}$. The $i$-th term of the sequence $\Delta_{u,j}$ is denoted by $\Delta_{u,j}(i)$. \\
Following \cite{Shi_Watson}, we divide $(\Delta_{u,j})_{j\geq 1}$ into (disjoint) classes corresponding to the truncation level of the children they represent. For that, let $q$ be a positive integer and define a new family $\Delta^q_{u,j}:=(\Delta^{q}_{u,j}(l))_{l\geq 1}$ as follows: $\Delta^q_{u,j}=\emptyset$ if the set $\{i\geq 2;\; \Delta_{u,j}(i)\in[q-1,q)\}$ is empty. Otherwise, the family $\Delta^q_{u,j}$ is made up of all the $\Delta_{u,j}(i)$, for $i\geq 2$, such that $\Delta_{u,j}(i)\in[q-1,q)$. By definition, $\Delta^q_{u,j}$ contains a finite number of elements that we rearrange in an increasing order. In order to see $\Delta^q_{u,j}$ as an element of $\mathcal{Q}$, we fill its tail with $\infty$. One can see that $\Delta^q_{u,j}=\emptyset$ for all $q>a$. By doing this, we will be able to index the branchings of all particles and that indexing will remain as $a$ gets larger, thus ensuring consistency when $a$ goes to infinity (see Definition \ref{DefBranchingSubor}).  \\
Let $k\geq 1$ be an integer. If there exists at least one integer $\mathfrak{p}\geq k$ such that during the first $\mathfrak{p}$ branchings of the particle $u$, the latter one gives progeny, exactly $k$ times, to at least one child at distance $d_{u,j}\in[q-1,q)$, $1\leq j\leq k$, from the pre-branching position of particle $u$, then we denote by $B^q_u(k)$ the smallest $\mathfrak{p}$ such that this event occurs. In other words, $B^q_u(0)=0$ and for any positive integer $k$
\begin{align*}
    B^q_{u}(k):=\inf\Big\{\mathfrak{p}\geq k;\; \sum_{j=1}^{\mathfrak{p}}\un_{\{\Delta^{q}_{u,j}\not=\emptyset\}}=k\Big\}.
\end{align*}
Let $\bm{p}=(q,k,i)\in(\N^*)^3$. The birth-time $\mathfrak{b}^a_{u\bm{p}}$ of the child $u\bm{p}$ of particle $u$ is set to be $\infty$ if $\mathfrak{b}^a_u=\infty$, $B^q_{u}(k)=\infty$ or $\Delta^{q}_{u,B^q_{u}(k)}(i)=\infty$ and we set $Y^a_{u\bm{p}}=\infty$. In particular, $Y^a_{u\bm{p}}=\infty$ when $q>a$. Otherwise, it is given by
\begin{align*}
    \mathfrak{b}^a_{u\bm{p}}:=T_{u,B^q_{u}(k)}.
\end{align*}
In the latter case, we introduce, for any $t\geq\mathfrak{b}^a_{u\bm{p}}$
\begin{align}
Y^a_{u\bm{p}}(t):=Y^a_u\big(\mathfrak{b}^a_{u\bm{p}}-\big)+\Delta^{q}_{u,B^q_{u}(k)}(i)+\xi_{u\bm{p}}(t-\mathfrak{b}^a_{u\bm{p}})+\sum_{j=1}^{N_{u\bm{p}}([\mathfrak{b}^a_{u\bm{p}},t])}\Delta_{u\bm{p},j}(1). \label{defyaup}
\end{align}
This yields generation $n+1$ and so on. \\
Let $u=u_1\cdots u_{\mathfrak{j}}\in\mathcal{U}$ for some $\mathfrak{j}\geq 1$. We extend the definition of $Y^a_{u}$ on $[0,\mathfrak{b}^a_u)$ in the following way: if $\mathfrak{j}\geq 2$, then for any $i\in\{1,\ldots,\mathfrak{j}-1\}$, $Y^a_u(t)=Y^a_{u_1\ldots u_{\mathfrak{i}}}(t)$ if $t\in[\mathfrak{b}^a_{u_1\cdots u_{\mathfrak{i}}},\mathfrak{b}^a_{u_1\cdots u_{\mathfrak{i}+1}})$ and $Y^a_u(t)=Y^a_{\varnothing}(t)$ if $t\in[0,\mathfrak{b}^a_{u_1})$. Otherwise, that is $\mathfrak{j}=1$, we set $Y^a_u(t)=Y^q_{\varnothing}(t)$ for any $t\in[0,\mathfrak{b}^a_u)$. In other words, if the particle $u$ is not born yet at time $t$, then $Y^a_u(t)$ stands for the position at time $t$ of the most recent ancestor alive (at time $t$).
\begin{defi}[$a$-truncated branching subordinator]\label{TroncBranchingSubor}
For any $t\geq 0$, introduce
\begin{align*}
    \mathcal{N}^a(t):=\{u\in\mathcal{U};\; \mathfrak{b}^a_u\leq t\},
\end{align*}
the set particles alive at time $t$ and define the random point measure
\begin{align*}
    Y^a(t):=\sum_{u\in\mathcal{N}^a(t)}\delta_{Y^a_u(t)}.
\end{align*}
$Y^a$ is referred to as a $a$-truncated branching subordinator. By a slight adaptation of Theorem 1.1(ii) in \cite{Bertoin_Mallein2019} we get that, for any $\lambda\geq 0$, the identity \eqref{ExpEspTruncature} holds true for $Y^a$ and $Y^a$ is a branching subordinator with finite birth intensity and characteristics $(d,\Lambda^a)$ in the sense of Section 4.2 of \cite{Bertoin_Mallein2019}.
\end{defi}
One can see that \eqref{ExpEspTruncature} implies $\E[\sum_{u\in\mathcal{N}^a(t)}1]<\infty$ for any $a> a_0(\Lambda)$. In particular, the number of particles alive at time $t$ is finite almost surely. Note that, since killing is not allowed in our study (see equation \eqref{condbranchinglp0}), we have that $\mathfrak{b}^a_u\leq t$ implies $Y^a_u(t)<\infty$ almost surely. 

\begin{defi}[Canonical trajectories associated with $Y^a$]\label{DefCanoTraj}
For any $t\geq 0$ and any $u\in\mathcal{U}$ such that $\mathfrak{b}^a_u<\infty$, define
\begin{align*}
    \bm{\xi}^a_u(t):=Y^{a}_u(\mathfrak{b}^a_u+t)-Y^{a}_u(\mathfrak{b}^a_u),
\end{align*}
the canonical trajectory of particle $u$. Let $\bm{\xi}^a_u$, $u\in\mathcal{U}$ such that $\mathfrak{b}^a_u=\infty$, be\textrm{ i.i.d }copies of $Y^a_{\varnothing}$ and independent of $(\bm{\xi}^a_u;\; u\in\mathcal{U},\;\mathfrak{b}^a_u<\infty)$. Hence it is easily seen from \eqref{defyaempty} and \eqref{defyaup} that, for all $a> a_0(\Lambda)$, $(\bm{\xi}^a_u)_{u\in\mathcal{U}}$ is a collection of\textrm{ i.i.d } subordinators with Laplace exponent $\phi$ and starting position $0$.
\end{defi}

\vspace{0.1cm}

\noindent\textbf{Step 2}: the branching subordinator $Y$

\noindent Let $a$ and $a'$ be two positive integers such that $a_0(\Lambda)<a\leq a'$. We define $((\mathcal{N}^{a'}(t))^{a})_{t\geq 0}$ to be the $a$-truncation of $(\mathcal{N}^{a'}(t))_{t\geq 0}$. The process $((\mathcal{N}^{a'}(t))^{a})_{t\geq 0}$ is obtained from $(\mathcal{N}^{a'}(t))_{t\geq 0}$ by, at each birth time, removing any born child (together with his future lineage) such that the distance between its birth-position and the position of its parent before giving progeny is larger or equal to $a$. We then introduce
\begin{align*}
    (Y^{a'})^a(t):=\sum_{u\in(\mathcal{N}^{a'}(t))^a}\delta_{Y^{a'}_u(t)}.
\end{align*}

\begin{lemma} \label{restrictionprop}
There exists a sequence $\{ Y^a; a\geq \lfloor a_0(\Lambda) \rfloor + 1 \}$ such that, for any $a\geq \lfloor a_0(\Lambda) \rfloor + 1$, $Y^a$ is a branching subordinator with finite birth intensity $(d,\Lambda^a)$ and, for any integer $a'\geq a$
\begin{align*}
    (Y^{a'})^a=Y^a\;\;\textrm{ almost surely}.
\end{align*}
\end{lemma}
\noindent We refer to Lemma 2.3 in \cite{Shi_Watson} for a proof. For any $t\geq 0$, $\{ Y^a(t); a\geq \lfloor a_0(\Lambda) \rfloor + 1 \}$ is an non-decreasing sequence of elements of $\mathcal{Q}$, which we recall is a subset of the space of point measures on $[0,\infty)$.
\begin{defi}[Branching subordinators]\label{DefBranchingSubor}
The process $Y=(Y(t))_{t\geq 0}$ defined by
\begin{align*}
    Y(t):=\lim_{a\to\infty}\uparrow	Y^a(t)
\end{align*}
is called a branching subordinator with characteristics $(d,\Lambda)$. Besides
\begin{align*}
    \mathcal{N}(t):=\bigcup_{a\geq \lfloor a_0(\Lambda) \rfloor + 1}\mathcal{N}^a(t)
\end{align*}
stands for the set of particles in $Y$ alive at time $t$. We also set $\mathcal{N}(t-):=\bigcup_{s \in (0,t)}\mathcal{N}(s)$.
\end{defi}

Note that for any $t\geq 0$, almost surely, $Y(t)$ belongs to $\mathcal{Q}$. Indeed, we have by construction, for any $q>a_0(\Lambda)$ and $a\geq q$, 
\begin{align*}
    \E\Big[\sum_{u\in\mathcal{N}^a(t)}\un_{\{Y^a_u(t)<q\}}\Big]=\E\Big[\sum_{u\in\mathcal{N}^q(t)}\un_{\{Y^q_u(t)<q\}}\Big],
\end{align*}
thus giving
\begin{align*}
    \E\Big[\sum_{u\in\mathcal{N}(t)}\un_{\{Y_u(t)<q\}}\Big]=\lim_{a\to\infty}\E\Big[\sum_{u\in\mathcal{N}^a(t)}\un_{\{Y^a_u(t)<q\}}\Big]=\E\Big[\sum_{u\in\mathcal{N}^q(t)}\un_{\{Y^q_u(t)<q\}}\Big]\leq \E\Big[\sum_{u\in\mathcal{N}^q(t)}1\Big]<\infty, 
\end{align*}
where the finiteness comes from the discussion after Definition \ref{TroncBranchingSubor}.
\begin{remark}
    To simplify, the construction of branching subordinators is presented for integers levels of truncation. However, one can easily extend the definition to any increasing sequence of truncation $(l_n)_{n\geq 0}$ such that $l_0>a_0(\Lambda)$ and $\lim_{n\to\infty}l_n=\infty$. Besides, by monotony, the limit (as in Definition \ref{DefBranchingSubor}) does not not depend on the choice of the sequence of truncation, that is for two different truncation sequences, the limits are almost surely equal. This allows us to use different conveniently chosen truncation sequences in our proofs.
\end{remark}

\begin{remark}[Strong branching property]\label{ReguBS}
    The branching subordinator $Y$ associated with the couple $(d,\Lambda)$ has a càdlàg version in $\mathcal{Q}$ and satisfies the branching property \eqref{BranchingProperty}. As a consequence, a stronger version of this branching property is verified by $Y$: recall that $\mathcal{F}^{Y}$ is the natural filtration of $Y$. For any $\mathcal{F}^{Y}$-stopping time $T$, for any $s\geq 0$, setting $\bm{y}=(y_n)_{n\geq 1}=Y(T)$, 
\begin{align*}
    Y(T+s)\overset{(\mathrm{d})}{=}\sum_{j\geq 1}\tau_{y_j}Y^{(j)}(s),
\end{align*}
where $(Y^{(j)}(s)$;\;$j\geq 1)$ is a collection of\textrm{ i.i.d }copies of $Y(s)$ and independent of $\mathcal{F}^{Y}_T$. \\ 
One can also note that for any $a>a_0(\Lambda)$, the law of the branching subordinator with finite birth intensity $Y^a$ is characterized by the couple $(d,\Lambda_a)$. Hence, the couple $(d,\Lambda)$ characterizes the law of the branching subordinator $Y$. We refer to \cite{Bertoin_Mallein2019} for a proof.
\end{remark}

We now prove a technical lemma that will allow us to define without ambiguity the leftmost particle within some sets. 
 \begin{lemma}[distinct locations of particles] \label{uniqueleftmost}
Let $\Lambda$ be a measure on $\mathcal{Q}$ satisfying \eqref{condbranchinglp0}, \eqref{condbranchinglplight} and Assumption \ref{stablelike} for some $\alpha \in (0,1)$. Let also $a>a_0(\Lambda)$. Then, for any fixed $t>0$, in the branching subordinator $Y^a$ with characteristic couple $(0,\Lambda^a)$, we have almost surely that there are finitely many particles in the system at time $t$ and their positions are distinct. As a consequence, the set of times where the positions of particles are not distinct has null Lebesgue measure. 
%Moreover, almost surely, the set of times $t$ such that there is not a unique $u \in \mathcal{N}(t)$ satisfying $Y^a_u(t)=\underline{Y^a}(t)$ has null Lebesgue measure. 
\end{lemma}
 Note that Assumption \ref{stablelike} plays an important role in the above lemma since, in the example of the branching Poisson process from Section \ref{branchingpoisson}, several particle can occupy the same position. 
\begin{proof}
The finiteness comes from the fact that, for any fixed $t$, $\sharp \mathcal{N}^a(t) <\infty$ almost surely (see the discussion after Definition \ref{TroncBranchingSubor}). We now prove that positions of particles are distinct. We have $\phi(\lambda)=\phi^a(\lambda)+\int_{\mathcal{Q}^{\star}} \left ( 1-e^{-\lambda x_1} \right ) \Lambda^a(\mathrm{d}\bm{x})$ by \eqref{seconddefphi} and \eqref{defphia}. Since $\Lambda^a(\mathcal{Q}^{\star})<\infty$ (see a little before \eqref{defphia}) the second term is bounded as $\lambda$ goes to infinity, so we get from Assumption \ref{stablelike} that $\phi^a(\lambda)$ is unbounded. 
By \eqref{defphia}, if we had $\Lambda^a_1((0,\infty))<\infty$, then $\phi^a(\lambda)$ would be bounded. We thus have $\Lambda^a_1((0,\infty))=\infty$. By \cite[Thm 27.4]{Sato}, the law of the subordinator with L\'evy measure $\Lambda^a_1$ is continuous at any fixed $t$. Using that $(\xi_u)_{u\in\mathcal{U}}$ are \textrm{ i.i.d }subordinators with L\'evy measure $\Lambda^a_1$, the independence of $(\xi_u)_{u\in\mathcal{U}}$ and $(\mathcal{P}_u)_{u\in\mathcal{U}}$, and the definition of $(Y^a_{u}(\cdot))_{u\in\mathcal{U}}$ we get that almost surely, at time $t$, the random variables in the collection $\{ Y^a_{u}(t); u \in \mathcal{N}^a(t) \}$ (resp. $\{ Y^a_{u}(t-); u \in \mathcal{N}^a(t-) \}$) take pairwise distinct values. The last claim follows from the first claim and Fubini's theorem. Indeed, If we denote by $A(t)$ (resp. $\tilde A(t)$) the event where the collection $\{ Y^a_{u}(t); u \in \mathcal{N}^a(t) \}$ (resp. $\{ Y^a_{u}(t-); u \in \mathcal{N}^a(t-) \}$) fails to take pairwise distinct values, then $\mathbb{E}[\int_0^{\infty}\mathds{1}_{A(t)} \mathrm{d}t]=\int_0^{\infty}\mathbb{P}(A(t)) \mathrm{d}t=\int_0^{\infty}0 \mathrm{d}t=0$ (resp. $\mathbb{E}[\int_0^{\infty}\mathds{1}_{\tilde A(t)} \mathrm{d}t]=0$). 
%Moreover these collections are all finite almost surely since, almost surely, the set $\mathcal{N}^a(t)$ is non-decreasing in $t$ and finite at all times $t$ (see Definition \ref{TroncBranchingSubor} and after). We thus get that, with probability one, at all times $\mathfrak{b}^a_{v}<\infty$, there is a unique $u \in \mathcal{N}^a(\mathfrak{b}^a_{v}-)$ such that $Y^a_{u}(\mathfrak{b}^a_{v}-)$ is minimal. This concludes the proof. 
%of the first point. The same argument as above shows the almost sure uniqueness of the leftmost particle at any fixed $t$. Integrating with respect to $t$ the probability of non-uniqueness over $[0,\infty)$ and using Fubini's theorem, we get that the Lebesgue measure of the set of times for which there is not uniqueness has null expectation. This concludes the proof. 
\end{proof}

We end this section by proving Lemma \ref{branchingleftmost} from Section \ref{lowestpartupperbound}. The proof is an easy consequence of the following fact on Poisson point processes. 
\begin{fact} \label{factpoisson}
Let $\mathcal{I}$ be a subinterval of $[0,\infty)$, $\mathcal{V}$ be a finite set equipped with its counting measure $\Sigma$ and $\mathcal{A}$ be a measured set equipped with a finite measure $\nu$. We consider a Poisson point process $\mathcal{P}$ on $\mathcal{I} \times \mathcal{A} \times \mathcal{V}$ with intensity measure $\mathrm{d}t \otimes \nu(\mathrm{d}a) \otimes \Sigma(\mathrm{d}v)$. Now let $S \subset \mathcal{I}$ have null Lebesgue measure and $f$ be a measurable function from $\mathcal{I} \setminus S$ to $\mathcal{V}$. We then define a point process $\mathcal{P}'$ on $\mathcal{I} \times \mathcal{A}$ by $\mathcal{P}' = \{ (t,a); \ t \notin S, (t,a,f(t)) \in \mathcal{P} \}$. Then $\mathcal{P}'$ is a Poisson point process on $\mathcal{I} \times \mathcal{A}$ with intensity measure $\mathrm{d}t \otimes \nu(\mathrm{d}a)$. 
\end{fact}

\begin{proof}[Proof of Lemma \ref{branchingleftmost}]
Recall the elements introduced in Step 1 of the above construction of $Y^a$. For $u\in\mathcal{U}$, let us denote by $\mathcal{P}^k_u$ the restriction of the Poisson point process $\mathcal{P}_u$ to $[0,2^k]$. Let $\mathcal{G}_k$ be the sigma field generated by $\{ (\xi_u)_{u\in\mathcal{U}}, (\mathcal{P}^k_u)_{u\in\mathcal{U}} \}$. Recall the notations $\mathcal{Q}^{\star}$ from \eqref{condbranchinglp0} and $\mathcal{U}_k^a$, $S_{br}^a$, $u(t)$, $\bm{x}(t)$, $S_{low,k}$ from Section \ref{lowestpartupperbound} and note that the random set $\mathcal{U}_k^a$ is measurable with respect to $\mathcal{G}_k$. By the construction from Step 1 we get that, on the event $\{\sharp \mathcal{U}_k^a \geq 1\}$, the conditional distribution of the point process $\{ (t,\bm{x}(t),u(t)); \ t \in [2^k,2^{k+1}) \cap S^{a}_{br}, u(t) \in \mathcal{U}_k^a \}$ (where the conditioning is with respect to $\mathcal{G}_k$) is that of Poisson point process on $[2^k,2^{k+1}) \times \mathcal{Q}^{\star} \times \mathcal{U}_k^a$ with intensity measure $\mathrm{d}t \otimes \Lambda^a(\cdot\cap\mathcal{Q}^{\star}) \otimes \Sigma$, where $\Sigma$ denotes the counting measure on $\mathcal{U}_k^a$. Let $S$ denote the set of times $t \in [2^k,2^{k+1})$ where the leftmost particle, among all particles in $\mathcal{U}_k^a$, is not unique, i.e. $S := \{ t \in [2^k,2^{k+1}); \ \sharp \{ u \in \mathcal{U}_k^a; Y^a_{u}(t-) = \min_{v \in \mathcal{U}_k^a} Y^a_{v}(t-) \} \geq 1 \}$. We also define a function $f$ from $[2^k,2^{k+1}) \setminus S$ to $\mathcal{U}$ that, at each time, assigns the index of the leftmost particle, among all particles in $\mathcal{U}_k^a$, at that time, i.e. $f(t):=u$ such that $Y^a_{u}(t-) = \min_{v \in \mathcal{U}_k^a} Y^a_{v}(t-)$ (therefore the range of $f$ almost surely lies in $\mathcal{U}_k^a$). Thanks to Assumption \ref{nojumpatbranchings}, the last term in the right-hand side of \eqref{defyaup} is null almost surely. Therefore, the collection of trajectories $\{ (Y^a_{u}(t))_{t \geq 2^k}; u \in \mathcal{U}_k^a \}$ is measurable with respect to $\mathcal{G}_k$ and so is the random couple $(S,f)$. Moreover, Lemma \ref{uniqueleftmost} tells us that, almost surely, $S$ has null Lebesgue measure. Therefore, applying Fact \ref{factpoisson}, we get that, almost surely on $\{\sharp \mathcal{U}_k^a \geq 1\}$, the conditional distribution of the point process $\{ (t,\bm{x}(t)); \ t \in S_{low,k} \}$ (where the conditioning is with respect to $\mathcal{G}_k$) has the claimed distribution. Since the conditional distribution does not depend on the particular realization in $\{\sharp \mathcal{U}_k^a \geq 1\}$, the result follows. 
\end{proof}

\section{Linear behavior and justification of \eqref{linearbehaviour0} and \eqref{linearbehaviour}} \label{prooflinearbehav}

In this appendix, we justify that, for any $d\geq0$ and for any measure $\Lambda$ on $\mathcal{Q}$ satisfying 
\eqref{condbranchinglp0} and \eqref{condbranchinglplight}, the behavior of the position of the leftmost particle in a branching subordinator $Y=(Y(t))_{t\geq 0}$ with characteristics $(d,\Lambda)$ is at most linear, in the sense that $\underline{Y}(t)=O(t)$ as $t\to\infty$ (note that it is not necessarily the case for trajectories of classical subordinators). We also discuss some cases where the growth of $t\mapsto\underline{Y}(t)$ is exactly linear. \\
Recall that $(Y(n))_{n\geq 0}$ is a branching random walk on the non-negative real line such that $Y(0)=0$ and it is supercritical, see Remark \ref{Survival}. Let us first assume that \eqref{condbranchinglp3} holds for some $\theta\geq 0$. It is well known since the famous works of Hammersley \cite{Hammersley}, Kingman \cite{Kingman} and  Biggins \cite{Biggins} that, if $-\infty<\kappa(\theta_0)<0$ for some $\theta_0>0$ (allowing $\kappa(0)=-\infty$), then, $\P$-almost surely
\begin{align*}
    \lim_{n \rightarrow \infty} \frac{\underline{Y}(n)}{n}=\sup_{\lambda\geq\theta_0}\frac{\kappa(\lambda)}{\lambda}=\sup_{\lambda>0}\frac{\kappa(\lambda)}{\lambda}\in[0,\infty),
\end{align*}
where the non-negativity comes from $\kappa(\cdot)$ being non-decreasing. In fact, the above result holds on the set of non-extinction $\{\forall n\geq 0,\; \mathcal{N}_{n}\not=\varnothing\}$, but as we have already mentioned previously, the condition $\Lambda(\{\emptyset\})=0$ in \eqref{condbranchinglp0} says there is no killing, that is, by construction, $\P(\forall n\geq 0,\; \mathcal{N}_{n}\not=\varnothing)=1$. Using that $t\mapsto\underline{Y}(t)$ is almost surely a non-decreasing function, we immediately deduce \eqref{linearbehaviour}. \\
One can see by construction that if we only require \eqref{condbranchinglplight}, then for any $a>a_0(\Lambda)$ (see Section \ref{moredef}), $\underline{Y}(t)\leq\underline{Y}^a(t)$ and $Y^a$ is a branching subordinator with characteristics $(d,\Lambda^a)$ such that $\Lambda^a$ satisfies \eqref{condbranchinglp3} for all $\theta\geq 0$. Moreover, $\kappa_a(0)=\Lambda(\{\emptyset\})-\int_{\mathcal{Q}}(\sum_{k\geq 2}\un_{\{x_k<a\}})\Lambda(\mathrm{d}\bm{x})=-\int_{\mathcal{Q}}(\sum_{k\geq 2}\un_{\{x_k<a\}})\Lambda(\mathrm{d}\bm{x})<0$ for all $a>a_0(\Lambda)$. Since $\kappa_a(\cdot)$ is continuous, one can always find $\theta_{a,0}>0$ such that $-\infty<\kappa_a(\theta_{a,0})<0$. Hence, thanks to \eqref{linearbehaviour}, we have $\limsup_{t\to\infty}\underline{Y}(t)/t\leq\sup_{\lambda\geq\theta_{a,0}}\kappa_a(\lambda)/\lambda=\sup_{\lambda>0}\kappa_a(\lambda)/\lambda$ for all $a>a_0(\Lambda)$, that is, by \eqref{laplaceexpospecialcase2} 
\begin{align*}
   \limsup_{t\to\infty}\frac{\underline{Y}(t)}{t}\leq d+\inf_{a>a_0(\Lambda)}\sup_{\lambda>0}\frac{1}{\lambda}\Big(\int_{\mathcal{Q}}\big(1-e^{-\lambda x_1}\big)\Lambda(\mathrm{d}\bm{x})-M_a(\lambda)\Big)\in[d,\infty).
\end{align*}
Still from the construction of the branching subordinator in Appendix \ref{construction}, each particle in the system moves with a drift $d$ and has only positive jumps and branchings with non-negative displacements. Therefore all particles in the system at time $t$ have a position larger or equal to $dt$ so $\liminf_{t \rightarrow \infty} \underline{Y}(t)/t\geq d$. Hence
\begin{align*}
    d\leq\liminf_{t\to\infty}\frac{\underline{Y}(t)}{t}\leq\limsup_{t\to\infty}\frac{\underline{Y}(t)}{t}\leq d+\inf_{a>a_0(\Lambda)}\sup_{\lambda>0}\frac{1}{\lambda}\Big(\int_{\mathcal{Q}}\big(1-e^{-\lambda x_1}\big)\Lambda(\mathrm{d}\bm{x})-M_a(\lambda)\Big)\in[d,\infty), 
\end{align*}
and this yields \eqref{linearbehaviour0}. In particular, if $d>0$, then the growth of $t\mapsto\underline{Y}(t)$ is linear.

\section{An equivalent formulation of Assumption \ref{regbranchoriginal}} \label{equivassump}

In this section we consider the following assumption. 
\begin{assum} \label{regbranch}
There exists $\sigma>0$ and $c_1>c_2>0$ such that 
\begin{align}
\limsup_{a \rightarrow \infty} M_a(c_1 a^{\sigma}) <\infty, \label{regbranchlow} \\
\liminf_{a \rightarrow \infty} \frac{\log(M_a(c_2 a^{\sigma}))}{a^{1+\sigma}} >0. \label{regbranchhigh}
\end{align}
\end{assum}
Let us notice from \eqref{condbranchinglplight} and dominated convergence that, for any fixed $a\geq 0$, $M_a(\lambda)$ remains bounded as $\lambda$ goes to infinity. We also notice from \eqref{noexpomom} that, for any fixed $\lambda\geq 0$, $M_a(\lambda) \rightarrow \infty$ as $a$ goes to infinity. Therefore, for any increasing function $f:\mathbb{R}_+\rightarrow\mathbb{R}_+$, $M_a(f(a))$ will be large (resp. bounded) as $a$ goes to infinity if $f$ increases slowly enough (resp. quickly enough). Assumption \ref{regbranch} can be understood as saying that the transition between the two possible behaviors of $M_a(f(a))$ is sharp enough. The following lemma proves the equivalence between Assumption \ref{regbranch} and Assumption \ref{regbranchoriginal}. 
%in the sense that there exists $\sigma>0$ and two positive constants $c_1,c_2$ with $c_1>c_2$ such that $M_u(c_1 u^{\sigma})$ is bounded as $u$ goes to infinity, and $M_u(c_2 u^{\sigma})$ goes to infinity fast. 

\begin{lemma} \label{equivassumption}
For any measure $\Lambda$ on $\mathcal{Q}$ satisfying \eqref{condbranchinglp0}-\eqref{condbranchinglplight} let $\mu$ be defined by \eqref{defmeasmu} and let $M_a(\lambda)$ be defined by \eqref{laplaceexpospecialcase2}. 
Then Assumption \ref{regbranch} holds for some $\sigma>0$ if and only if Assumption \ref{regbranchoriginal} holds for the same $\sigma$. 
\end{lemma}

\begin{proof}
We first assume that Assumption \ref{regbranch} holds true for some $\sigma>0$ and $c_1>c_2>0$, and we prove \eqref{assump2meas}. Since $M_a(\lambda)=\int_{[0,a)}e^{-\mathfrak{z}\lambda}\mu(\mathrm{d}\mathfrak{z})$, we have 
\begin{align}
\forall a>0, \lambda \geq 0, \ \mu([0,a))\leq e^{\lambda a} M_a(\lambda). \label{trivialbound0}
\end{align}
Choosing $\lambda=c_1 a^{\sigma}$ in \eqref{trivialbound0} and using \eqref{regbranchlow} we get the finiteness of the limsup in \eqref{assump2meas}. Then, again by definition of $M_a(\lambda)$, 
\begin{align}
\forall a>0, \lambda \geq 0, \ \log(M_a(\lambda)) \leq \log(\mu([0,a))). \label{trivialbound}
\end{align}
Choosing $\lambda=c_2 a^{\sigma}$ in \eqref{trivialbound} and using \eqref{regbranchhigh} we get the positivity of the liminf in \eqref{assump2meas}, concluding the proof of \eqref{assump2meas}. \\
We now assume that Assumption \ref{regbranchoriginal} holds true for some $\sigma>0$ and we prove \eqref{regbranchlow} and \eqref{regbranchhigh}. For any $a>0$ and $\lambda \geq 0$, we have $M_a(\lambda) \leq \sum_{n=0}^{\lfloor a \rfloor} e^{-\lambda n}\mu([n,n+1])$. Choosing $\lambda=c_1 a^{\sigma}$ for some $c_1$ strictly larger than the limsup in \eqref{assump2meas}, we get $M_a(c_1 a^{\sigma}) \leq \sum_{n=0}^{\lfloor a \rfloor} e^{-c_1 a^{\sigma} n} \mu([n,n+1]) \leq \sum_{n=0}^{\infty} e^{-c_1 n^{1+\sigma}} \mu([n,n+1])$. This sum is finite by \eqref{assump2meas} so \eqref{regbranchlow} follows. Then, \eqref{trivialbound0} translates into $\log(M_a(\lambda)) \geq -\lambda a +\log(\mu([0,a)))$. Choosing $\lambda=c_2 a^{\sigma}$ for some $c_2$ strictly smaller than the liminf in \eqref{assump2meas}, we get $\log(M_a(c_2 a^{\sigma})) \geq -c_2 a^{1+\sigma} +\log(\mu([0,a)))$. From this and \eqref{assump2meas}, we get \eqref{regbranchhigh}. 
\end{proof}
In light of the equivalence from Lemma \ref{equivassumption}, the double appearance of $\sigma$ in \eqref{regbranchhigh} is natural. Indeed, \eqref{regbranchlow} implies the finiteness of the limsup in \eqref{assump2meas} which in turn implies, via \eqref{trivialbound}, that for any function $f(\cdot):[0,\infty)\rightarrow[0,\infty)$ we have $\limsup_{a \rightarrow \infty} \log(M_a(f(a)))/a^{1+\sigma} <\infty$. Therefore, \eqref{regbranchhigh} only says that a counterpart of \eqref{regbranchlow} is satisfied.
\end{appendix}

\begin{merci}
The authors are grateful to Zhan Shi and Bastien Mallein for interesting discussions, advises, and careful reading. This work is supported by Beijing Natural Science Foundation, project number IS24067. The first author benefited during the preparation of this paper from the support of Hua Loo-Keng Center for Mathematical Sciences (AMSS, Chinese Academy of Sciences) and from the National Natural Science Foundation of China (No. 12288201).
\end{merci}

\bibliographystyle{alpha}
\bibliography{BiblioAK_GV}

\end{document}